\documentclass[11pt]{amsart}
\usepackage{amsthm,amsmath,amsxtra,amscd,amssymb,xypic}
\usepackage[final]{graphicx}
\usepackage[all]{xy}
\usepackage{url}

\makeatletter
\let\@wraptoccontribs\wraptoccontribs
\makeatother

\newcommand{\ssigma}{{\boldsymbol{\Sigma}}}

\newcommand{\pr}{\operatorname{Pr}}

\newcommand{\rc}{{\operatorname{rc}}}

\newcommand{\stratum}[2][{g}]{\beta_{#1}({#2})}

\newcommand{\Spec}{\operatorname{Spec}}
\newcommand{\Stab}{\operatorname{Stab}}
\newcommand{\Proj}{\operatorname{Proj}}

\newcommand{\AQ}{{\mathbb{A}}}
\newcommand{\CC}{{\mathbb{C}}}

\newcommand{\HH}{{\mathbb{H}}}

\newcommand{\EE}{{\mathbb{E}}}

\newcommand{\QQ}{{\mathbb{Q}}}
\newcommand{\RR}{{\mathbb{R}}}

\newcommand{\ZZ}{{\mathbb{Z}}}

\newcommand{\VV}{{\mathbb{V}}}
\newcommand{\bS}{{\mathbb{S}}}

\newcommand{\calH}{{\mathcal H}}

\newcommand{\calI}{{\mathcal I}}
\newcommand{\calA}{{\mathcal A}}
\newcommand{\calC}{{\mathcal C}}
\newcommand{\calL}{{\mathcal L}}
\newcommand{\calM}{{\mathcal M}}

\newcommand{\calX}{{\mathcal X}}

\newcommand{\calS}{{\mathcal S}}

\newcommand{\op}{\operatorname}
\newcommand{\M}[1]{\calM_{#1}}
\newcommand{\Mb}[1]{\overline{\calM}_{#1}}
\newcommand{\ab}[1][g]{\calA_{#1}}

\newcommand{\BS}[1][g]{{\calA_{#1}^{\op {BS}}}}
\newcommand{\Sat}[1][g]{{\calA_{#1}^{\op {Sat}}}}
\newcommand{\Vor}[1][g]{{\calA_{#1}^{\op {Vor}}}}
\newcommand{\Perf}[1][g]{{\calA_{#1}^{\op {Perf}}}}
\newcommand{\Centr}[1][g]{{\calA_{#1}^{\op {Centr}}}}
\newcommand{\Matr}[1][g]{{\calA_{#1}^{\op {Matr}}}}

\newcommand{\Tor}[1][g]{{\calA_{#1}^{\op {tor}}}}
\newcommand{\STor}[1][g]{{\tilde{\calA}_{#1}^{\op {tor}}}}

\newcommand{\Asigma}[1][g]{{{\calA}_{#1}^\ssigma}}
\newcommand{\XBB}{X^{\op{BB}}}
\newcommand{\XBS}{X^{\op{BS}}}
\newcommand{\XRBS}{X^{\op{RBS}}}
\newcommand{\XTor}{X^{\op{tor}}}
\newcommand{\Sp}{\op{Sp}}

\newcommand{\GSp}{\op{GSp}}
\newcommand{\GSpin}{\op{GSpin}}
\newcommand{\spin}{\op{spin}}
\newcommand{\GL}{\op{GL}}
\newcommand{\SL}{\op{SL}}
\newcommand{\Spin}{\op{Spin}}
\newcommand{\SO}{\op{SO}}
\newcommand{\Or}{\op{O}}
\newcommand{\Mat}{\op{Mat}}
\newcommand{\U}{\op{U}}
\newcommand{\Sym}{\op{Sym}}
\newcommand{\Hom}{\op{Hom}}
\newcommand{\Lie}{\op{Lie}}

\newcommand{\Aut}{\op{Aut}}

\newcommand{\CH}{\op{CH}}

\newcommand{\ch}{\op{ch}}
\newcommand\cusp{\op{cusp}}

\newcommand\diag{\operatorname{diag}}
\newcommand\tr{{\rm tr}}
\newcommand\codim{{\rm codim}}
\newcommand\rank{\operatorname{rank}}

\newcommand{\pu}{\bullet}
\newcommand{\cohloc}[3]{H^{#1}(#2,#3)}
\newcommand{\coh}[2][\pu]{\cohloc {#1}{#2}{\QQ}}
\newcommand{\icohloc}[3]{IH^{#1}(#2,#3)}
\newcommand{\icoh}[2][\pu]{\icohloc {#1}{#2}{\QQ}}

\newcommand{\cohvi}[3][\pu]{\cohloc {#1}{#2}{\VV_{#3}}}
\newcommand{\cohcloc}[3]{H_c^{#1}(#2,#3)}
\newcommand{\cohinloc}[3]{H_!^{#1}(#2,#3)}
\newcommand{\cohtwoloc}[3]{H_{(2)}^{#1}(#2,#3)}
\newcommand{\cohcusploc}[3]{H_{\cusp}^{#1}(#2,#3)}
\newcommand{\cohc}[2][\pu]{\cohcloc {#1}{#2}{\QQ}}
\newcommand{\cohcvi}[3][\pu]{\cohcloc {#1}{#2}{\VV_{#3}}}
\newcommand{\cohinvi}[3][\pu]{\cohinloc {#1}{#2}{\VV_{#3}}}
\newcommand{\cohtwovi}[3][\pu]{\cohtwoloc {#1}{#2}{\VV_{#3}\otimes \CC}}
\newcommand{\cohcuspvi}[3][\pu]{\cohcusploc {#1}{#2}{\VV_{#3}\otimes \CC}}

\newcommand{\BMloc}[3]{\bar H_{#1}(#2,#3)}
\newcommand{\BM}[2][\pu]{\BMloc {#1}{#2}{\QQ}}

\theoremstyle{plain}
\newtheorem{thm}{Theorem} %

\newtheorem{prop}{Proposition}
\newtheorem{cor}{Corollary}

\newtheorem{qu}{Question}

\theoremstyle{definition}
\newtheorem{df}{Definition}
\newtheorem{rem}{Remark}

\newtheorem{exa}{Example}

\usepackage{color}
\definecolor{forestgreen}{rgb}{0.13, 0.55, 0.13}

\title[Topology of $\ab$]{The topology of $\ab$ and its compactifications}
\author{Klaus Hulek}
\address{Institut f\"ur Algebraische Geometrie, Leibniz Universit\"at Hannover, Welfengarten 1, 30060 Hannover, Germany}
\email{hulek@math.uni-hannover.de}
\author{Orsola Tommasi}
\address{Mathematical Sciences, Chalmers University of Technology and the University of Gothenburg, SE-412 96 G\"oteborg, Sweden}
\email{orsola@chalmers.se}

\contrib[with an appendix by]{Olivier Ta\"{\i}bi}
\address{Unit\'e de math\'ematiques pures et appliqu\'ee, ENS de Lyon, 69364 Lyon Cedex 07, France}
\email{olivier.taibi@ens-lyon.fr}

\begin{document}
\begin{abstract}
We survey old and new results about the cohomology of the moduli space  $\ab$ of principally polarized abelian varieties of genus $g$ and its compactifications. The main emphasis lies on the 
computation of the cohomology for small genus and on stabilization results. We review both geometric and representation theoretic approaches to the problem. The appendix provides a detailed discussion of 
computational methods based on trace formulae and automorphic representations,
in particular Arthur's endoscopic classification of automorphic representations
for symplectic groups.
\end{abstract}

\maketitle

\section{Introduction}

The study of moduli spaces of abelian varieties goes back as far as the late 19th century when Klein and Fricke studied families of elliptic curves.
This continued in the 20th century with the work of Hecke. The theory of higher dimensional abelian varieties was greatly influenced by C.L. Siegel  
 who studied automorphic forms in several variables. In the 1980's Borel and others started a systematic study of the topology of locally symmetric spaces 
 and thus also moduli spaces of abelian varieties. From 1977 onwards  Freitag, Mumford and Tai proved groundbreaking results on the geometry of Siegel modular varieties.  
 Since then a vast body of literature has appeared on abelian varieties and their moduli.

One of the fascinating aspects of abelian varieties is that the subject is at the crossroads of several mathematical fields: geometry, arithmetic, topology and representation theory.
In this survey we will restrict ourselves to essentially one aspect, namely the topology of the moduli space $\ab$ of principally polarized abelian varieties and   its compactifications. 
This in itself is a subject which has been covered in numerous research papers and several survey articles. Of the latter we would like to mention 
articles by van der Geer and Oort \cite{vdgoosurvey},  Sankaran and the first author \cite{husageometry},  van der Geer \cite{vdgsurvey}, Grushevsky \cite{grAgsurvey} and van der Geer's contribution
to the Handbook of Moduli \cite{vdghomsurvey}. 
Needless to say that all of these articles concentrate on different aspects and include new results as progress was made. In this article we will, naturally, recall some of the 
basic ideas of the subject, but we will in particular concentrate on two aspects. One is the actual {\em computation} of the cohomology of $\ab$ and its compactifications in small genus. 
The other aspect is the phenomenon of {\em stabilization} of cohomology, which means that in certain ranges the cohomology groups  do not depend on the genus.  
One of our aims is to show how concepts and techniques from such different fields as algebraic geometry, analysis, differential geometry, representation theory and topology   
come together fruitfully in this field to provide powerful tools and results.  

In more detail, we will cover the following topics: In Section \ref{sec:analyticapproach} we will set the scene and introduce the moduli space of principally polarized abelian varieties $\ab$
as an analytic space.  In Section \ref{sec:tautologicalring} we introduce the tautological ring of $\ab$. Various compactifications of $\ab$ will be introduced and discussed
in Section \ref{sec:proportionality}, where we will also recall the proportionality principle.   In Section \ref{sec:L2cohomology} we shall recall work on $L^2$-cohomology,  Zucker's conjecture
and some results from representation theory.
This will mostly be a recapitulation of more classical results, but the concepts and the techniques from this section will play a major role in the final two sections of this survey.
In Section \ref{sec:smallgenus} we will treat the computation of the cohomology   in low genus in some detail. In particular, we will discuss the cohomology of both $\ab$ itself, but also of 
its various compactifications, and we will treat both singular and intersection cohomology.  Finally, stabilization is the main topic of  Section \ref{sec:stabilization}. Here we not only treat the 
classical results, such as Borel's stabilization theorem for $\ab$ and its extension by Charney and Lee to the Satake compactification $\Sat$, but we will also discuss recent work of 
Looijenga and Chen as well as stabilization of the cohomology for (partial) toroidal compactifications.

In the appendix, by Olivier Ta\"ibi, we explain how the Arthur--Selberg trace formula can be harnessed
to explicitly compute the Euler characteristic of certain local systems on $\ab$
and their intermediate extensions to $\Sat$, i.e.\ $L^2$-cohomology by Zucker's
conjecture. Using Arthur's endoscopic classification and an inductive procedure
individual $L^2$-cohomology groups can be deduced. An alternative computation
uses Chenevier and Lannes' classification of automorphic cuspidal
representations for general linear groups having conductor one and which are
algebraic of small weight. Following Langlands and Arthur, we give details for
the computation of $L^2$-cohomology in terms of Arthur--Langlands parameters,
notably involving branching rules for (half-)spin representations.

Throughout this survey we will work over the complex numbers $\CC$. We will also restrict to moduli of principally polarized abelian varieties, although the same questions can be asked more generally for 
abelian varieties with other polarizations, as well as for abelian varieties with extra structure such as complex or real multiplication or level structures. This restriction is mostly due to lack of space, but also to the fact that,
in particular, moduli spaces with non-principal polarizations have received considerably less attention. 

\medskip
    {\small {\bf Acknowledgements.} We are grateful to Dan Petersen for very useful comments on an earlier draft of this paper.
      The second author would like to acknowledge support from her \emph{Research Award 2016} of the Faculty of Science of the University of Gothenburg during the preparation of this paper.
      The first author is grateful to the organizers of the {\em Abel symposium 2017} for a wonderful conference. 
      }

\section{The complex analytic approach}\label{sec:analyticapproach}

As we said above the construction of the moduli space $\ab$ of principally  polarized abelian varieties (ppav) of dimension $g$ can be approached from 
different angles:  it can be constructed algebraically as the underlying coarse moduli space of the {\em moduli stack} 
of principally polarized abelian varieties \cite{fachbook} or analytically as a locally symmetric domain \cite{bila}. The algebraic approach results in a smooth
Deligne--Mumford stack defined over $\Spec(\ZZ)$ of dimension $g(g+1)/2$, the analytic construction gives a normal complex analytic space with finite quotient singularities.
The latter is, by the work of Satake \cite{satake} and Baily--Borel \cite{babo} in fact a quasi-projective variety.   

Here we recall the main facts about the analytic approach. The {\em Siegel upper half plane} is defined as the space of symmetric $g \times g$ matrices with positive definite
imaginary part
\begin{equation}
\HH_g= \{ \tau \in \Mat(g \times g, \CC ) \mid \tau = {}^t \tau, \Im(\tau) > 0 \}. 
\end{equation}
This is a homogeneous domain. To explain  this we consider the standard symplectic form
\begin{equation}
J_g=\left(
\begin{array}{cc}
0          &  {\bf 1}_g\\
-{\bf 1}_g &   0
\end{array}
\right).
\end{equation}
The {\em real symplectic group }Ê $\Sp(2g,\RR)$ is the group fixing this form: 
\begin{equation}
\Sp(2g,\RR)=\{M \in \GL(2g,\RR) \mid {}^tMJM=J \}.
\end{equation}
Similarly we define $\Sp(2g,\QQ)$ and $\Sp(2g,\CC)$.

The discrete subgroup
\begin{equation*}
\Gamma_g=\Sp(2g,\ZZ)
\end{equation*} 
will be of special importance for us.
The group of {\em (complex) symplectic similitudes} is defined by 
\begin{equation}
\GSp(2g,\CC)=\{M \in \GL(2g,\CC) \mid {}^tMJM=cJ {\mbox { for some }} c\in \CC^* \}.
\end{equation}

The real symplectic group $\Sp(2g,\RR)$ acts on the Siegel space $\HH_g$ from the left  by 
\begin{equation}
M= \left(
\begin{array}{cc}
A          &  B\\
C &   D 
\end{array}
\right) 
: \tau \mapsto (A\tau +B)(C\tau + D)^{-1}. 
\end{equation}
Here $A,B,C,D$ are $g \times g$  matrices. This action is transitive and the stabilizer of the point $i{\bf 1}_g$ is 
\begin{equation}
\Stab(i{\bf 1}_g)=\left\{M \in \Sp(2g,\RR) \mid 
 M= \left(
\begin{array}{cc}
A  &  B\\
-B &   A
\end{array}\right)
\right\}. 
\end{equation}
The map
\begin{equation*}
  \left(
\begin{array}{cc}
A   &  B\\
-B &   A
\end{array}
\right) \mapsto A+iB
\end{equation*}
defines an isomorphism
\begin{equation*}
\Stab(i{\bf 1}_g) \cong \U(g)
\end{equation*}
where $\U(g)$  is the unitary group. This is the maximal compact subgroup of $\Sp(2g,\RR)$,
and in this way we obtain a description of the Siegel space as a homogeneous domain
\begin{equation}
\HH_g \cong \Sp(2g,\RR) / \U(g).
\end{equation}

The involution 
\begin{equation*}
\tau \mapsto - \tau^{-1}
\end{equation*}
defines an involution with $i{\bf 1}_g$ an isolated fixed point. Hence $\HH_g$ is a symmetric homogeneous domain.

The object which we are primarily interested in is the quotient 
\begin{equation}
\ab = \Gamma_g \backslash \HH_g.
\end{equation}
The discrete group $\Gamma_g=\Sp(2g,\ZZ)$ acts properly discontinuously on $\HH_g$ and hence $\ab$ is a normal analytic space with finite 
quotient singularities. This is a coarse moduli space for principally polarized abelian varieties (ppav), see \cite[Chapter 8]{bila}.  Indeed, given a 
point $[\tau]\in \ab$ one obtains a ppav explicitly as $A_{[\tau]}=\CC^g/L_{\tau}$, where $L_{\tau}$ is the lattice in $\CC^g$ spanned by the columns 
of the $(g \times 2g)$-matrix $(\tau,{\bf 1}_g)$.    

There are various variations of this construction.
One is that one may want to describe (coarse) moduli spaces of abelian varieties with polarizations which are non-principal. This is achieved by 
 replacing the standard symplectic form given by $J_g$ by
 \begin{equation}
J_g=\left(
\begin{array}{cc}
0          &  D\\
- D &   0
\end{array}
\right)
\end{equation}
where $D=\diag(d_1, \ldots, d_g)$ is a diagonal matrix  and the entries $d_i$ are positive integers with $d_1 | d_2 | \cdots | d_g$. 

Another variation involves the introduction of {\em level structures}. 
This results in choosing suitable finite index subgroups of 
$\Gamma_g$. Here we will only consider the {\em principal congruence subgroups of level $\ell$}, which are defined
by 
\begin{equation*}
\Gamma_g(\ell)=\{ g\in \Sp(2g,\ZZ) \mid g \equiv {\bf 1}\mod \ell  \}. 
\end{equation*} 
The quotient
\begin{equation}
\ab(\ell) = \Gamma_g(\ell) \backslash \HH_g
\end{equation}
parameterizes ppav with a level-$\ell$ structure. The latter is the choice of a symplectic basis of the group $A[\ell]$ of $\ell$-torsion points on an abelian variety $A$. Recall that 
$A[\ell] \cong (\ZZ/ \ell \ZZ)^{2g}$ and that $A[\ell]$ is equipped with a natural symplectic form, the {\em Weil pairing}, see \cite[Section IV.20]{mumfordbook}. If $\ell \geq 3$ then the group $\Gamma_g(\ell)$ acts 
freely on $\HH_g$, see e.g. \cite{serre-rigidity}, \cite[Corollary 5.1.10]{bila},
and hence $\ab(\ell)$ is a complex manifold (smooth quasi-projective variety).
In particular, analogously to the case of $\HH_g$ discussed in \cite[\S8.7]{bila}, for $\ell\geq 3$ the manifold $\ab(\ell)$ 
carries an honest universal family $\calX_g(\ell) \to \ab(\ell)$, which can be defined as the quotient
\begin{equation}\label{equ:universal}
{\mathcal X}_g(\ell)=\Gamma_g(\ell)\ltimes \ZZ^{2g} \backslash \HH_g\times \CC^g.
\end{equation}
Here the semidirect product $\Gamma_g(\ell)\ltimes \ZZ^{2g}$ is defined by the action of $\Sp(2g,\ZZ)$ on $\ZZ^{2g}$
and the action is given by
\begin{equation}
(M,m,n): (\tau, z) \mapsto (M(\tau),  ((C\tau+D)^{t})^{-1}z + \tau m + n ).
\end{equation}
for all $M\in\Gamma_g(\ell)$ and  $m,n\in\ZZ^g$,
The map  $\calX_g(\ell) \to \ab(\ell)$ is induced by the projection $\HH_g\times \CC^g \to \HH_g$.
The universal family $\calX_g(\ell)$ makes sense also for $\ell=1,2$ if we define it as an orbifold quotient. This allows to define a universal family $\calX_g:=\calX_g(1) \rightarrow \ab$ on $\ab$.

A central object in this theory is the  {\em Hodge bundle} $\EE$. Geometrically, this is given by associating to each point $[\tau]\in \ab$ the cotangent space of the abelian variety $A_{[\tau]}$ at the origin.
This gives an honest vector bundle over the level covers $\ab(\ell)$ and an orbifold vector bundle over $\ab$. In terms of automorphy factors this can be written
as
\begin{equation}
\EE := \Sp(2g,\ZZ) \backslash  \HH_g \times \CC^g
\end{equation}
given by 
\begin{equation}
M: (\tau,v) \mapsto (M(\tau), (C\tau +D)v)
\end{equation}
for $M \in \Sp(2g,\ZZ)$.
 
As we explained above, the Siegel space $\HH_g$ is a symmetric homogeneous domain and as such has a {\em compact dual}, namely the {\em symplectic Grassmannian}
\begin{equation}   
Y_g=  \{ L \subset \CC^{2g} \mid \dim L=g, J_g|_L\equiv 0 \}.
\end{equation}
This is a homogeneous projective space of complex dimension $g(g+1)/2$. In terms of algebraic groups it can be identified with
\begin{equation}
Y_g= \GSp(2g,\CC)/Q
\end{equation}
where 
\begin{equation}
Q=\left\{ \left(\begin{array}{cc}
A   &  B\\
C &   D 
\end{array}
\right)
 \in \GSp(2g,\CC) \mid  C=0 \right\}
\end{equation}
is a {\em Siegel parabolic subgroup}.

The Siegel space $\HH_g$ is the open subset of $Y_g$ of all maximal isotropic subspaces $L \in Y_g$ such that the restriction of the symplectic form $-i J_g|_L$ is positive definite.  
Concretely, one can associate to $\tau\in \HH_g$ the subspace spanned by the rows of the matrix $(- {\bf 1}_g, \tau)$.  
The cohomology ring $H^{\bullet}(Y_g,\ZZ)$ is very well understood in terms of Schubert cycles. 
Moreover $Y_g$ is a smooth rational variety and the cycle map defines  an isomorphism $\CH^{\bullet}(Y_g) \cong  H^{\bullet}(Y_g,\ZZ)$ between the {\em Chow ring} and 
the cohomology ring of $Y_g$. 
For details we refer the reader to van der Geer's  survey paper \cite[p. 492]{vdghomsurvey}. 

Since $Y_g$ is a Grassmannian we have a tautological sequence of  vector bundles 
\begin{equation}\label{lem:basicexact} 
0 \to E \to H \to Q \to 0
\end{equation}
where $E$ is the tautological subbundle, $H$ is the trivial bundle of rank $2g$ and $Q$ is the tautological quotient bundle. In particular, the fibre $E_L$ at a point $[L] \in Y_g$ is the isotropic
subspace $L\subset \CC^{2g}$.    
We denote the Chern classes of $E$ by
\begin{equation*}
u_i := c_i(E),
\end{equation*}
which we can think of as elements in Chow or in the cohomology ring. The exact sequence (\ref{lem:basicexact}) immediately gives the relation
\begin{equation}\label{equ:basicrel}
(1 + u_1 + u_2 + \ldots + u_g)(1 - u_1 + u_2 - \ldots + (-1)^gu_g)=1.
\end{equation}
Note that this can also be expressed in the form
 \begin{equation}\label{equ:Chernchar}
\ch_{2k}(E)=0,  \,  k\geq 1
\end{equation}
where $\ch_{2k}(E)$ denotes the degree $2k$ part of the Chern character.

\begin{thm}\label{thm:tautringcompactdual}
The classes $u_i$ with $i= 1, \dots, g$ generate $\CH^{\bullet}(Y_g) \cong  H^{\bullet}(Y_g,\ZZ)$ and all relations are generated by 
the relation
\begin{equation*}
(1 + u_1 + u_2 + \ldots + u_g)(1 - u_1 + u_2 - \ldots + (-1)^gu_g)=1.
\end{equation*}
\end{thm}

\begin{df}\label{def:Rg}
By $R_g$ we denote the abstract graded ring generated by elements $u_i; i = 1, \ldots,g$ subject to  relation (\ref{equ:basicrel}).
\end{df}

In particular, the dimension of $R_g$ as a vector space is equal to $2^g$.
As a consequence of Theorem \ref{thm:tautringcompactdual} and the definition of $R_g$ we obtain
\begin{prop}
The intersection form on $H^{\bullet}(Y_g,\ZZ)$ defines a perfect pairing on $R_g$. 
The ring $R_g$ is a Gorenstein ring with socle $u_1u_2 \ldots u_g$. Moreover there are natural isomorphisms
\begin{equation*}
R_g/(u_g) \cong R_{g-1}.
\end{equation*}
\end{prop}
As a vector space $R_g$ is generated by $\prod_iu_i^{\varepsilon_i}$ with $\varepsilon_i \in \{0,1\}$ and the duality is given by $\varepsilon_i \mapsto  1 - \varepsilon_i$.

\section{The tautological ring of $\ab$}\label{sec:tautologicalring}

We have already encountered  the {\em Hodge bundle} $\EE$ on $\ab$ in Section \ref{sec:analyticapproach}. There we defined it as  $\EE= \pi_*(\Omega^1_{\calX_g/\ab})$ where $\pi: \calX_g \to \ab$ is the universal abelian variety. 
As we pointed out this is an orbifold bundle or, alternatively, an honest vector bundle on $\ab(\ell)$ for $\ell \geq 3$, as can be seen from the construction of the  
universal family $\calX_g(\ell) \to \ab(\ell)$ given in (\ref{equ:universal}).
We use the following notation for  the Chern classes
\begin{equation} 
\lambda_i=c_i(\EE).
\end{equation}
We can view these in either Chow or cohomology (with rational coefficients). Indeed, in view of the fact that the group $\Sp(2g,\ZZ)$ does not act freely, we will from now on mostly
work with Chow or cohomology with rational coefficients. 

There is also another way in which the Hodge bundle can be defined.
Let us recall that it can be realized explicitly as the quotient of the trivial bundle $\HH_g\times \CC^g$ on which the group $\Sp(2g,\ZZ)$ acts by
\begin{equation} 
M=\left(
\begin{array}{cc}
A          &  B\\
C &   D 
\end{array}
\right) : (\tau,v) \mapsto (M(\tau), (C\tau +D)v).
\end{equation}

One can also consider this construction in two steps. First, one considers the embedding of $\HH_g$ into its compact dual  $Y_g = GSp(2g,\CC)/Q$ as explained in Section~\ref{sec:analyticapproach}.
Secondly, one can prove that $\EE$ coincides with the quotient of the restriction to $\HH_g$ of the tautological subbundle $E$ defined in \eqref{lem:basicexact}, by the natural $\Sp(2g,\ZZ)$-action.

As explained in \cite[\S~13]{vdgsurvey}, this is a special case of a construction that associates with every complex representation of $\GL(g,\CC)$ a homolomorphic vector bundle on $\ab$.  
This construction is very  important in the theory of modular forms and we will come back to it (in a slightly different guise) in Section \ref{sec:proportionality} below.

\begin{df}
The {\em tautological ring} of $\ab$ is the subring defined by the classes $\lambda_i, i=1, \ldots, g$. We will use this both in the Chow ring $\CH_{\QQ}^{\bullet}(\ab)$ or in cohomology $H^{\bullet}(\ab,\QQ)$. 
\end{df}
The main properties of the tautological ring can be summarized  by the following  
\begin{thm}\label{thm:tautological}
The following holds in $\CH_{\QQ}^{\bullet}(\ab)$:
\begin{itemize}
\item[\rm (i)]
$(1 + \lambda_1 + \lambda_2 + \ldots + \lambda_g)(1 - \lambda_1 + \lambda_2 - \ldots + (-1)^g\lambda_g)=1$
\item[\rm (ii)] $\lambda_g=0$
\item[\rm (iii)] There are no further relations between the $\lambda$-classes on $\ab$ and hence the tautological ring of $\ab$ is isomorphic to $R_{g-1}$. 
\end{itemize}
The same is true in $H^{\bullet}(\ab,\QQ)$.
\end{thm}
\begin{proof}
We refer the reader to van der Geer's paper \cite{vdgeercycles}, where the above statements appear as Theorem 1.1, Proposition 1.2 and 
Theorem 1.5  respectively. This is further discussed in \cite{vdghomsurvey}.
\end{proof}

\section{Compactifications and the proportionality principle}\label{sec:proportionality}

The space $\ab$ admits several compactifications which are geometrically relevant. The smallest compactification is the Satake compactification $\Sat$, which is a special case of the Baily--Borel compactification for
locally symmetric domains.   Set-theoretically this is simply the disjoint union 
\begin{equation}\label{eq:strata-satake}
\Sat = \ab \sqcup \ab[g-1] \sqcup \ldots \sqcup \ab[0].
\end{equation}
It is, however, anything but trivial to equip this with a suitable topology and an analytic structure. This can be circumvented by using modular forms.
A {\em modular form} of {\em weight} $k$ is a holomorphic function
\begin{equation*}
f: \HH_g \to \CC
\end{equation*}
such that for every $M \in \Sp(2g,\ZZ)$ with 
$M=\left(
\begin{array}{cc}
A          &  B\\
C &   D
\end{array} \right)$
the following holds:
\begin{equation*}
f(M(\tau))= \det(C\tau + D)^kf(\tau). 
\end{equation*}
In terms of the Hodge bundle modular forms of weight $k$ are exactly the sections of the $k$-fold power of the determinant of the Hodge bundle $\det(\EE)^{\otimes k}$.
If $g=1$, then we must also add a growth condition on $f$ which ensures holomorphicity at infinity, for $g\geq 2$ this condition is automatically satisfied. 
We denote the space of all modular forms of weight $k$ with respect to the full modular group $\Gamma_g$  by $M_k(\Gamma_g)$. This is a finite dimensional vector space.
Using other representations of $\Sp(2g,\CC)$ one can generalize this concept to {\em vector valued Siegel modular forms}. 
For an introduction to 
modular forms we refer the reader to \cite{freitagbooksiegel}, \cite{vdgsurvey}. 
The spaces $M_k(\Gamma_g)$ form a graded ring
$\oplus_{k\geq 0} M_k(\Gamma_g)$ and one obtains 
\begin{equation}  
\Sat = \Proj  \oplus_{k \geq 0}M_k(\Gamma_g).
\end{equation}
Indeed, one can take this as the definition of $\Sat$. The fact that the graded algebra of modular forms is finitely generated implies that $\Sat$ is a projective variety. It contains $\ab$ as a Zariski open subset, thus providing
$\ab$ with the structure of a quasi-projective variety.  We say that a modular form is a {\em cusp form} if its restriction to the {\em boundary} of $\Sat$, by which we mean the complement of $\ab$ in $\Sat$, vanishes. 
The space of cusp forms of weight $k$ 
is denoted by $S_k(\Gamma_g)$. 

The Satake compactification $\Sat$ is naturally associated to $\ab$. However, it has the disadvantage that it is badly singular along the boundary.
This can  be remedied by considering toroidal compactifications $\Tor$ of $\ab$. These compactifications were introduced by Mumford, following ideas of Hirzebruch on the resolution of surface singularities. We refer the reader  to the 
standard book by 
Ash, Mumford, Rapoport and Tai \cite{amrtbook}. Toroidal compactifications depend on choices, more precisely we need an admissible collection of admissible fans. In the case of principally polarized abelian varieties,
this reduces to the choice of one admissible fan $\Sigma$ covering the rational closure    $\Sym^2_{\rc}(\RR^g)$ of the space  $\Sym_{>0}^2(\RR^g)$ of positive definite symmetric $g \times g$-matrices.   
To be more precise, an {\em admissible fan} $\Sigma$ is a collection of rational polyhedral cones lying in  $\Sym^2_{\rc}(\RR^g)$, with the following properties: it is closed under taking intersections and faces, the union of these cones 
covers   $\Sym^2_{\rc}(\RR^g)$ and the collection is invariant under the natural action of $\GL(g,\CC)$ on  $\Sym^2_{\rc}(\RR^g)$ with the additional property that there are only finitely many $\GL(g,\CC)$-orbits of such cones.   
The construction of such fans is non-trivial and closely related to the reduction theory of quadratic forms. 

There are three classical decompositions (fans) of  $\Sym^2_{\rc}(\RR^g)$ that have all been well studied and whose associated toroidal compactifications are by now reasonably well understood, namely the {\em second Voronoi}, 
the  {\em perfect cone or first Voronoi} and the {\em central cone decomposition}, leading to the compactifications $\Vor$, $\Perf$ and $\Centr$ respectively.   The Voronoi compactification $\Vor$ has a
{\em modular} interpretation due to Alexeev \cite{alexeev} and Olsson \cite{olsson}. Indeed, Alexeev introduced the notions of stable semi-abelic varieties and semi-abelic pairs, for which he constructed a moduli stack. It turns out that this is
in general  not irreducible
and that so-called {\em extra-territorial} components exist. The space $\Vor$ is the normalization of the principal component of the coarse moduli scheme associated to Alexeev's functor.
In contrast to this, Olsson's construction uses logarithmic geometry to give the principal component $\Vor$ directly.
The perfect cone or first Voronoi compactification $\Perf$ is very interesting from the  point of view of the {\em Minimal Model Program (MMP)}. Shepherd-Barron \cite{shepherdbarron} has shown that $\Perf$ is
a $\QQ$-factorial variety  with canonical singularities if  $g\geq 5$ and that  its canonical divisor is nef if $g \geq 12$, in other words $\Perf$  is, in this range, a canonical model in the sense of MMP.
We refer the reader also to  \cite{arsb} where some missing arguments from  \cite{shepherdbarron} were completed.
Finally, the central cone compactification $\Centr$ coincides with the {\em Igusa blow-up} of the Satake compactification $\Sat$ \cite{igusadesing}. 

All toroidal compactifications admit a natural morphism $\Tor \to \Sat$ which restrict to  the identity on $\ab$. 
A priori, a toroidal compactification need not be projective, but there is a projectivity criterion \cite[Chapter 4, \S2]{amrtbook}
which guarantees projectivity if the underlying decomposition $\Sigma$ admits a suitable  piecewise linear $\Sp(2g,\ZZ)$-invariant support function. All 
the toroidal compactifications discussed above are projective.
For the second Voronoi
compactification $\Vor$ it was only in \cite{alexeev} that the existence of a suitable support function was exhibited. 

For $g\leq 3$ the three toroidal compactifications described above coincide, but in general they are all different and none is a refinement of another. Although toroidal compactifications $\Tor$ behave better 
with respect to singularities than the Satake compactification $\Sat$, this does not mean that they are necessarily smooth. To start with, the coarse moduli space of $\ab$ is itself a singular variety due to the existence of abelian varieties with non-trivial 
automorphisms. These are, however, only finite quotient singularities and we can always avoid these by going to level covers of level $\ell \geq 3$. We refer to this situation as {\em stack smooth}.
For $g \leq 3$ the toroidal compactifications described above are also stack smooth, but this changes considerably for $g\geq 4$, when singularities do appear. A priori, the only property we know of these 
singularities is that they are (finite quotients of) toric singularities. For a discussion of the singularities of $\Vor$ and $\Perf$ see \cite{DutourHulekSchuermann}.  
By taking subdivisions of the cones we can for each toroidal compactification 
$\Tor$  obtain a smooth toroidal resolution  $\STor \to \Tor$. We shall refer to these compactifications as (stack) smooth toroidal compactifications, often dropping the word stack in this context.

It is  natural to ask whether the classes $\lambda_i$ can be extended to $\Sat$ or to  toroidal compactifications $\Tor$.
As we will explain later in Section~\ref{sec:stabilization},
it was indeed shown by Charney and Lee \cite{chle}
that the $\lambda$-classes can be lifted
to the Satake compactification $\Sat$ via the restriction map $H^{2i}(\Sat,\QQ) \to H^{2i}(\ab,\QQ)$. These lifts are, however, not canonical. Another lift was obtained by Goresky and Pardon \cite{GoreskyPardon},
working, however, with cohomology  with complex coefficients. Their classes are canonically defined and we denote them by   $\lambda_i^{\operatorname{GP}} \in H^{2i}(\Sat,\CC)$. 
It was recently shown by Looijenga \cite{lo-pardon}
that there are values of $i$ for which the Goresky--Pardon classes have a non-trivial imaginary part
and hence differ from the Charney--Lee classes. 
This will be discussed in more detailed in Theorem~\ref{thm:lo-pardon}.

The next question is whether the $\lambda$-classes can be extended to toroidal compactifications $\Tor$. By a  result of Chai and Faltings \cite{fachbook} the Hodge bundle $\EE$ can be extended to toroidal 
compactifications $\Tor$.  The argument is that  one can define a universal semi-abelian scheme over $\Tor$ and fibrewise one can then take the cotangent space at the origin. 
In this way we obtain extensions of the $\lambda$-classes in cohomology or in the operational Chow ring. 
Analytically, Mumford \cite{mumhirz}  proved  that one can extend the Hodge bundle as a vector bundle $\tilde {\EE}$  to any smooth toroidal compactification $\STor$.  Moreover, if $p: \STor \to \Sat$ 
is the canonical map, then by \cite{GoreskyPardon}  we have
\begin{equation*}
c_i(\tilde {\EE})= p^*(\lambda_i^{\operatorname{GP}}).
\end{equation*}
We also note the following: if $D$ is the (reducible)  boundary divisor in a level $\STor(\ell)$ with $\ell\geq 3$, then by \cite[p. 25]{fachbook}
\begin{equation*} 
\Sym^2(\tilde {\EE})\cong \Omega^1_{\STor(\ell)}(D).
\end{equation*}
In order to simplify the notation we denote the classes $c_i(\tilde {\EE})$ on $\STor$ also by $\lambda_i$. It is a crucial result that the basic relation (i) of Theorem \ref{thm:tautological} also extends to smooth toroidal
compactifications.
\begin{thm}\label{thm:relationChowcoho}
The following relation holds in $\CH_{\QQ}^{\bullet}(\STor)$:
\begin{equation}\label{equ:tautringtor}
(1 + \lambda_1 + \lambda_2 + \ldots + \lambda_g)(1 - \lambda_1 + \lambda_2 - \ldots + (-1)^g\lambda_g)=1.
\end{equation}
 \end{thm}
\begin{proof}
This was shown in cohomology by van der Geer \cite{vdgeercycles} and in the Chow ring by Esnault and Viehweg \cite{esvi}.
\end{proof}

As before we will define the {\em tautological} subring of the Chow ring  $\CH_{\QQ}^{\bullet}(\STor)$ (or of the cohomology ring  $H^{\bullet}(\STor,\QQ)$) as the subring generated by 
the (extended) $\lambda$-classes.  Now $\lambda_g \neq 0$ and we obtain the following
\begin{thm}\label{thm:tautring}
The tautological ring of $\STor$ is isomorphic to $R_g$.
\end{thm}
\begin{proof}
We first note that the relation (\ref{equ:tautringtor}) holds. The statement then follows since the 
intersection form defines a perfect pairing on the $\lambda$-classes. In particular we have
\begin{equation*}
\lambda_1\ldots\lambda_g=\frac{1}{(g(g+1))/2)!}\left(\prod_{j=1}^g(2j-1)!! \right)\lambda_1^{\frac12g(g+1)} \neq 0.
\end{equation*}  
\end{proof}
Indeed, one can think of the tautological ring as part of the cohomology contained in all (smooth) toroidal compactifications of $\ab$. Given  any two such toroidal compactifications one can always find 
a common smooth resolution and pull the $\lambda$-classes back to this space. In this sense  the tautological ring does not depend on a particular chosen compactification $\STor$.  

The top intersection numbers of the $\lambda$-classes can be computed explicitly by relating them to (known) intersection numbers on the compact dual.  This is a special case of the 
{\em Hirzebruch--Mumford proportionality}, which had first been found by Hirzebruch in the co-compact case \cite{hirzprop1}, \cite{hirzprop2}  and then been extended by Mumford \cite{mumhirz} to the non-compact case.

\begin{thm}  
The top intersection numbers of the $\lambda$-classes on a smooth toroidal compactification $\STor$  are proportional to the corresponding top intersection numbers of the Chern classes of the universal 
subbundle on the compact dual $Y_g$. More precisely if $n_i$ are non-negative integers with $\sum in_i = g(g+1)/2$,
then
\begin{equation*}
\lambda_1^{n_1}\cdot \ldots \cdot \lambda_g^{n_g}= (-1)^{\frac12g(g+1)}\frac{1}{2^g}\left(\prod_{j=1}^g\zeta(1-2j)\right)u_1^{n_1}\cdot \ldots \cdot u_g^{n_g}.
\end{equation*}
\end{thm}
As a corollary, see also the proof of Theorem \ref{thm:tautring}, we obtain 
\begin{cor}
\begin{equation*}
\lambda_1^{\frac12g(g+1)}=
(-1)^{\frac12g(g+1)}\frac{(g(g+1)/2)!}{2^{g}}\left(\prod_{k=1}^g\frac{\zeta(1-2k)}{(2k-1)!!}\right).
\end{equation*} 
\end{cor}
We note that the formula we give here  is the intersection number on the stack $\ab$, i.e. we take the involution given by $-\bf 1$ into account. In particular this means that  the degree of the 
Hodge line bundle on $\Sat[1]$ equals $1/24$.

This can also be rewritten in terms of Bernoulli numbers. Recall that the Bernoulli numbers  $B_j$ are  defined by the generating function
\begin{equation*}
\frac{x}{e^x-1}=\sum_{k=0}^{\infty}B_k\frac{x^k}{k!}, \quad |x| < 2\pi
\end{equation*}
and the relation between the Bernoulli numbers and the $\zeta$-function is given by  
\begin{equation*}
\zeta(-n)=(-1)^n\frac{B_{n+1}}{n+1}
\end{equation*}
respectively
\begin{equation*}
B_{2n}=(-1)^{n-1}\frac{2(2n)!}{(2\pi)^{2n}}\;\zeta(2n).
\end{equation*}

A further application of the Hirzebruch--Mumford proportionality is that it describes the growth behaviour of the dimension of the spaces of modular forms of weight $k$. 
\begin{thm}
The dimension of the space of modular forms of weight $k$ with respect to the group $\Gamma_g$  grows asymptotically as follows when $k$ is even and goes to infinity:
\begin{equation*}
\dim M_k(\Gamma_g) \sim 2^{\frac12(g-1)(g-2)} \ k^{\frac12g(g+1)} \;\prod_{j=1}^g\frac{(j-1)!}{(2j)!}(-1)^{j-1}B_{2j}.
\end{equation*}

\end{thm}
\begin{proof}
The proof consists of several steps. The first is to go to a level-$\ell$ cover and apply Riemann--Roch to the line bundle $(\det \tilde{\EE})^{\otimes k}$ as the modular forms of weight $k$ with respect to the principal congruence subgroup 
$\Gamma_g(\ell)$ are just the sections of this line bundle.
The second step is to prove that this line bundle has no higher cohomology. Consequently, the Riemann--Roch expression for $(\det \tilde{\EE})^{\otimes k}$ gives the dimension of the space of sections, and the leading term (as $k$ grows) is
determined by the self-intersection number $\lambda_1^{\frac12g(g+1)}$ on $\STor(\ell)$. This shows that
\begin{equation*}
 \dim M_k (\Gamma_g(\ell)) \sim 2^{-\frac12g(g+1)-g} \ k^{\frac12g(g+1)} \; [\Gamma_g(\ell): \Gamma_g] \  V_g \; \pi^{-\frac12g(g+1)}
 \end{equation*}
 where $V_g$ is Siegel's volume
 \begin{equation*}
 V_g=2^{g^2+1}\ \pi^{\frac12g(g+1)} \; \prod_{j=1}^g\frac{(j-1)!}{(2j)!}(-1)^{j-1}B_{2j}.
 \end{equation*}
 The third step is to descend to $\ab$ by applying the Noether--Lefschetz fixed-point formula. It turns out that this does not affect the leading term, with the exception of cancelling the index $[\Gamma_g(\ell): \Gamma_g]$. 
\end{proof}
This was used by Tai \cite{tai} in his proof that $\ab$ is of general type for $g\geq 9$. The same principle can be applied to compute the growth behaviour of the space of modular forms or cusp forms also in the 
non-principally polarized case, see e.g. \cite[Sect.~II.2]{husageometry}. Indeed, Hirzebruch--Mumford proportionality can also be used to study  other homogeneous domains, for example orthogonal modular varieties, see \cite{GHSorth1}, \cite{GHSorth2}.

\section{$L^2$ cohomology and Zucker's conjecture}\label{sec:L2cohomology}
\label{s:Ltwocohomology}

In the 1970's and 1980's great efforts were made to understand the cohomology of locally symmetric domains.  
In the course of this various cohomology theories were studied, notably {\em intersection cohomology}Ê and {\em $L^2$-cohomology}.
Here we will briefly recall some basic facts which will be of relevance for the discussions in the following sections.

One of the drawbacks of singular (co)homology is that Poincar\'e duality fails for singular spaces.   
It was one of the main objectives of Goresky and MacPherson 
to remedy this situation when they introduced intersection cohomology. Given a  space $X$ of real dimension $m$,
one of  the starting points of intersection theory is the choice of 
a good stratification 
\begin{equation*}
 X=X_m \supset X_{m-1} \supset \cdots \supset X_1 \supset X_0
\end{equation*}
by closed subsets $X_i$ such that each point $x \in X_i \setminus X_{i-1}$ has a neighbourhood $N_x$ which is again suitably stratified and whose homeomorphism type does not depend on $x$.
The usual singular $k$-chains are then replaced by chains which intersect each stratum $X_{m-i}$ in a set of dimension at most $k-i+p(k)$ where $p(k)$ is the perversity. This leads to 
the intersection homology groups $IH_k(X,\QQ)$ and dually to intersection cohomology  $IH^k(X,\QQ)$.   We  will restrict ourselves here mostly to (complex) algebraic 
varieties where  the strata $X_i$ have real dimension $2i$,  and we will work with  the middle perversity, which means that $p(k)=k-1$. Intersection  cohomology  not 
only satisfies Poincar\'e duality, but it also has many other good properties, notably we have a Lefschetz theorem and a K\"ahler package, including a Hodge decomposition. 
In case of a smooth manifold, or, more generally, a variety with locally quotient singularities, intersection (co)homology and singular (co)homology coincide.
The drawback is that 
intersection cohomology loses some of its functorial properties (unless one restricts to stratified maps) and that it is typically hard to compute it from first principles.   
Deligne later gave a sheaf-theoretic construction which is particularly suited to algebraic varieties. The main point is the construction of an intersection cohomology complex $\calI\calC_X$  whose 
cohomology gives $IH^{\bullet}(X,\QQ)$. Finally we mention the {\em decomposition theorem} which for a projective morphism $f: Y \to X$ relates the intersection cohomology of $X$ with that of $Y$.
For an introduction to intersection cohomology we refer the reader to \cite{KirwanWoolf}. An excellent exposition of the decomposition theorem can be found in \cite{deCataldoMigliorini}.

Although we are here primarily interested in $\ab$ and its compactifications, much of the technology employed here is not special to this case, but applies more generally to Hermitian symmetric domains
and hence we will now move our discussion into this more general setting.
Let $G$ be a connected reductive group which, for simplicity, we assume to be semi-simple, and let $K$ be a maximal compact subgroup. We also assume that $D=G/K$ carries a $G$-equivariant complex structure 
in which case we speak of an {\em Hermitian symmetric space}.  The prime example we have in mind is $G=\Sp(2g,\RR)$ and $K=U(g)$ in which case $G/K=\HH_g$. If $\Gamma \subset G(\QQ)$ is an arithmetic 
subgroup we consider the quotient
\begin{equation*}
X=\Gamma \backslash G/K
\end{equation*}
which is called a {\em locally symmetric} space.
In our example, namely for  $\Gamma=\Sp(2g,\ZZ)$, we obtain 
\begin{equation*}
\ab=\Sp(2g,\ZZ)\backslash \Sp(2g,\RR)/U(g)= Sp(2g,\ZZ)\backslash \HH_g.
\end{equation*}
As in the Siegel case,  we also have several compactifications in this more general setting. 
The first is the {\em Baily--Borel compactification} $\XBB$ which for $\ab$  is nothing but the Satake compactification $\Sat$. As in the Siegel case  it can be defined as the proj
of a graded ring of automorphic forms, which gives it the structure of a projective variety. Again as in the Siegel case, one can define toroidal compactifications $\XTor$  which are compact normal analytic spaces.
Moreover, there are two further topological compactifications, namely the {\em Borel--Serre} compactification $\XBS$ and the  {\em reductive Borel--Serre} compactification $\XRBS$ \cite{BorelSerre}.  These are topological 
spaces which do not carry an analytic structure. The space $\XBS$ is a manifold with corners that is homotopy equivalent to $X$, whereas $\XRBS$ is typically a very singular stratified space. More details on their construction and properties can be found in \cite{BorelJi}.
These spaces are related by maps
\begin{equation*} 
\XBS \to \XRBS \to \XBB  \leftarrow \XTor
\end{equation*}
where the maps on the left hand side of $\XBB$ are continuous maps and the map on the right hand side is an analytic map. All of these spaces have natural stratifications which are suitable for 
intersection cohomology. For a survey on this topic we refer the reader to \cite{Goreskysurvey} which we follow closely in parts. 

Another important cohomology theory is $L^2$-cohomology. For this one considers the space of square-integrable differential forms 
\begin{equation*}
\Omega_{(2)}^i(X)= \left\{ \omega \in \Omega^i_X \mid \int \omega \wedge *\omega  < \infty,   \int d\omega \wedge *d\omega  < \infty \right\}.
\end{equation*}
This defines the $L^2$-cohomology groups
\begin{equation*}
H_{(2)}^i(X)= \op{ker} d / \op{im} d.
\end{equation*}
These cohomology groups are representation theoretic objects and can be expressed in terms of relative group cohomology as follows, see \cite[Theorem 3.5]{borel-regularization}:
\begin{equation}\label{equ:representation}
H_{(2)}^i(X) =H^i(\mathfrak g, K;  L^2(\Gamma \backslash G)^\infty)
\end{equation}
where $ L^2(\Gamma \backslash G)^\infty$ is the module of $L^2$-functions on $\Gamma \backslash G$ such that all derivatives by $G$-invariant differential operators are square integrable. Indeed, this isomorphism holds not only for cohomology with coefficients  in $\CC$ but more generally 
for cohomology with values in local systems. The famous Zucker conjecture says that the $L^2$-cohomology of $X$ and the intersection cohomology of $\XBB$ are naturally isomorphic. This was proven 
independently by  Looijenga \cite{LoijZucker} and Saper and Stern \cite{saperstern} in the late 1980's:
\begin{thm}[Zucker conjecture] 
There is a natural isomorphism
\begin{equation*}
H_{(2)}^i(X) \cong IH^i(\XBB,\CC) \mbox{ for all }i \geq 0.
\end{equation*}
\end{thm}
In 2001 Saper \cite{saper1}, \cite{saper2} established another isomorphism namely
\begin{thm} 
There is a natural isomorphism
\begin{equation*}
IH^i(\XRBS,\CC) \cong IH^i(\XBB,\CC) \mbox{ for all }i \geq 0.
\end{equation*}
\end{thm}

We conclude this section with two results concerning  specifically the case of abelian varieties as they will be relevant for the discussions in the next sections. 
The following theorem is part of Borel's work on the stable cohomology of $\ab$, see \cite[Remark 3.8]{borel2}: 
\begin{thm} \label{teo:L2IHSat}
There is a natural isomorphism
\begin{equation*}
H_{(2)}^k(\ab)  \cong H^k(\ab,\CC) \mbox{ for all } k<g.
\end{equation*}
\end{thm}
Let us remark that one can also view this theorem as a consequence of Zucker's conjecture, since $\coh[k]{\ab}=\icoh[k]{\ab}$ and $\icoh[k]{\XBB}$ coincide in degree $k<g$ as a consequence of the fact that the codimension of the boundary of $X$ in $\XBB$ is $g$. 

Finally we notice  the following connection between the tautological ring $R_g$ and intersection cohomology of $\Sat$, see also  \cite{HGIH}:
\begin{prop}\label{prop:tautinih}
There is a natural inclusion $R_g \hookrightarrow IH^{\bullet}(\Sat,\QQ)$ of graded vector spaces of the tautological ring into the intersection cohomology of the Satake compactification $\Sat$.
\end{prop}
 \begin{proof}  
By the natural map from cohomology to intersection cohomology we can interpret the (extended) classes $\lambda_i$ on $\Sat$ 
as classes in $IH^{2i}(\Sat)$.  Via the decomposition theorem 
we have an embedding  $IH^{2i}(\Sat) \subset H^{2i}(\STor)$ where $\STor$ is a (stack) smooth toroidal compactification.  Since the classes $\lambda_i$ 
satisfy the relation of Theorem \ref{thm:relationChowcoho} we obtain a map from $R_g$ to  $IH^{\bullet}(\Sat,\QQ)$.
Since moreover the intersection pairing defines a perfect pairing on $R_g$ there can be no further relations among the classes $\lambda_i \in IH^{2i}(\Sat)$ and hence we have an embedding 
$R_g \hookrightarrow IH^{\bullet}(\Sat,\QQ)$. 

Alternatively, one can see this from the isomorphism \eqref{equ:representation} by looking at the decomposition of $H_{(2)}^i(X)$ induced by the decomposition of $L^2(\Gamma \backslash G)$ into $G$-representations. Then it is known that the trivial representation occurs with multiplicity one in $L^2(\Gamma \backslash G)$ and that its contribution coincides with $R_g$.
\end{proof}

 \begin{rem}
   We mentioned earlier as a motivation for the introduction of the tautological ring that it is contained in the cohomology of all smooth toroidal compactifications $\STor$ of $\ab$.
   Proposition~\ref{prop:tautinih} provides an explanation for this.
        Applying the decomposition theorem to the canonical map $\STor \to \Sat$ we find that 
$H^{\bullet}(\STor,\QQ)$ contains 
$IH^{\bullet}(\Sat,\QQ)$ as a subspace, which itself  contains the tautological ring $R_g$.
\end{rem}

\section{Computations in small genus}
\label{sec:smallgenus}

In this section, we consider a basic  topological invariant of $\ab$ and its compactifications, namely the  cohomology with $\QQ$-coefficients. As we work over the field of complex numbers, the cohomology groups will carry mixed Hodge structures (i.e. a Hodge and a weight filtration). We will describe the mixed Hodge structures whenever this is possible because of their geometric significance. In particular, we will denote by $\QQ(k)$ the pure Hodge structure of Tate of weight $-2k$.
If $H$ is a mixed Hodge structure, we will denote its Tate twists by $H(k):= H\otimes \QQ(k)$. We will also denote any extension
$$0\rightarrow B \rightarrow H \rightarrow A \rightarrow 0$$
of a pure Hodge structure $A$ by a pure Hodge structure $B$ by $H=A+B$.

Although its geometric and algebraic importance is obvious, the cohomology ring $\coh{\ab}$ is completely known in only surprisingly few cases. The cases $g=0,1$ are of course trivial.
The cases $g=2,3$ are special in that the  locus of jacobians is dense in $\ab$ in these genera. This can be used to obtain information on the cohomology ring $\coh{\ab}$ from the known descriptions of $\coh{\M g}$ for these values of $g$. For $g=2$  the Torelli map actually extends to an isomorphism from the Deligne--Mumford compactification $\Mb 2$ to the
(in this case canonical)  toroidal compactification of $\ab[2]$. This map identifies $\ab[2]$ with the locus of stable curves of compact type and from this one can easily obtain that the cohomology ring is isomorphic to the tautological ring in this case.

In general, however, the cohomology ring of $\ab$ is larger than the tautological ring. This is already the case for $g=3$. In \cite{hain}, Richard Hain computed the rational cohomology ring of $\ab[3]$ using techniques from Goresky and MacPherson's stratified Morse theory. His result is the following:

\begin{thm}[{\cite{hain}}]
The rational cohomology ring of $\ab[3]$ is isomorphic to 
$\QQ[\lambda_1])/(\lambda_1^4)$ in degree $k\neq 6$.
In degree $6$ it is given by a $2$-dimensional mixed Hodge structure which is an extension of the form $\QQ(-6)+\QQ(-3)$.
  \end{thm}
 
Let us remark that the class of the extension in $\coh[6]{\ab[3]}$ is unknown. Hain expects it to be given by a (possibly trivial) multiple of $\zeta(3)$. 

For genus up to three, also the cohomology of all compactifications we mentioned in the previous sections is known. For the Satake compactification, this result is due to Hain in \cite[Prop. 2 \& 3]{hain}:

\begin{thm}[\cite{hain}]
 The following holds:
  \begin{itemize}
    \item[(i)]
The rational cohomology ring of $\Sat[2]$ is isomorphic to 
$\QQ[\lambda_1])/(\lambda_1^4)$.
\item[(ii)]
The rational cohomology ring of $\Sat[3]$ is isomorphic to 
$\QQ[\lambda_1])/(\lambda_1^7)$ in degree $k\neq 6$. In degree $6$ it is given by a $3$-dimensional mixed Hodge structure which is an extension of the form $\QQ(-3)^{\oplus 2}+\QQ$.
\end{itemize}
  \end{thm}
 
Hain's approach is based on first computing the cohomology of the link of $\Sat[g-1]$ in $\Sat$ for $g=2,3$ and then using a Mayer--Vietoris sequence to deduce from this the cohomology of $\Sat$. Alternatively, one could obtain the same result by looking at the natural stratification \eqref{eq:strata-satake} of $\Sat$ and calculating the cohomology using the Gysin spectral sequence for cohomology with compact support, which in this case degenerates at $E_2$ and yields
$$
\coh{\Sat}\cong\bigoplus_{0\leq k\leq g}\cohc{\ab[k]},\ \ \ g=2,3.
$$
Here $\cohc{\ab}$ denotes cohomology with compact support.
We recall that $\ab$ is rationally smooth, so that we can obtain $\cohc{\ab}$ from $\coh{\ab}$ by Poincar\'e duality.

For $g\leq 3$ the situation is easy for toroidal compactifications as well. Let us recall that the commonly considered toroidal compactifications all coincide in this range, so that we 
can talk about {\em the} toroidal compactification in genus $2$ and $3$. As mentioned above, the compactification $\Tor[2]$ can be interpreted as the moduli space $\Mb 2$ of stable genus $2$ curves, whose rational cohomology was computed by Mumford in \cite{mumfordtowards}. The cohomology of $\Tor[3]$ can be computed using the Gysin long exact sequence in cohomology with compact support associated with its toroidal stratification. Using this, we proved in \cite{huto} that the cohomology of $\Tor[3]$ is isomorphic to the Chow ring of $\Tor[3]$, which is known by \cite{vdgeerchowa3}. The results are summarized by the following two theorems.
\begin{thm}\label{teo:cohomologA2}
For the toroidal compactification $\Tor[2]$  the cycle map defines an isomorphism $CH_{\QQ}^{\bullet}(\Tor[2]) \cong  H^{\bullet}(\Tor[2], \QQ)$. There is no odd dimensional cohomology and the even Betti numbers are given by
  $$  \begin{array}{c|cccc}
    i&0&2&4&6\\\hline
    b_i&1&2&2& 1
    \end{array}
  $$
\end{thm}

\begin{thm}\label{teo:cohomologA3}
For the toroidal compactification $\Tor[3]$  the cycle map defines an isomorphism $CH_{\QQ}^{\bullet}(\Tor[3]) \cong  H^{\bullet}(\Tor[3],\QQ)$. There is no odd dimensional cohomology and the even Betti numbers are given by
  $$  \begin{array}{c|ccccccc}
    i&0&2&4&6&8&10&12\\\hline
    b_i&1&2&4& 6 &4&2&1
    \end{array}
  $$
\end{thm}

We also note that due to van der Geer's results \cite{vdgeerchowa3} explicit generators of $H^{\bullet}(\Tor,\QQ)$ for $g=2,3$ are known, as well as the ring structure. 

The method of computing the cohomology of a (smooth) toroidal compactification using its natural stratification is sufficiently robust that it can be applied to $\Vor[4]$ as well. The key point here is that, 
although a general abelian fourfold is not the jacobian of a curve, the Zariski closure of the locus of jacobians in $\Sat[4]$ is an ample divisor $\overline{J}_4$. In particular, its complement $\Sat[4]\setminus \overline{J_4}=\ab[4]\setminus J_4$ is an affine variety of dimension $10$ and thus its cohomology groups vanish in degree $>10$. Hence, there is a range in degrees where the cohomology of compactifications of $\ab[4]$ can be determined from cohomological information on $\ab$ with $g\leq 3$ and on the moduli space $\M4$ of curves of genus $4$. 
Using Poincar\'e duality, this is enough to compute the cohomology of the second Voronoi compactification $\Vor[4]$, which is smooth, in all degrees different from the middle cohomology $H^{10}$. However, a single missing Betti number can always be recovered from the Euler characteristic of the space.

By the work of Ta\"\i{}bi, see Proposition (\ref{prop:eA4}) in the appendix,  it is now known that the Euler characteristics $e(\ab[4])=9$. This allows us to rephrase the results of \cite{huto2} as follows:

\begin{thm}\label{teo:cohomologyA4}
The following holds:
  \begin{itemize}
    \item[\rm (i)]
  The rational cohomology of $\Vor[4]$ vanishes in odd degree and is algebraic in all even degrees. The Betti numbers are given by
  $$  \begin{array}{c|ccccccccccc}
    i&0&2&4&6&8&10&12&14&16&18&20\\\hline
    b_i&1&3&5&11&17&19&17&11&5&3&1
    \end{array}
  $$
\item[\rm (ii)]
  The Betti numbers of $\Perf[4]$ in degree $\leq 8$ are given by
  $$
  \begin{array}{c|ccccccccc}
    i&0&1&2&3&4&5&6&7&8\\\hline
    b_1&1&0&2&0&4&0&8&0&14
    \end{array}
  $$
  and the rational cohomology in this range consists of Tate Hodge classes of weight $2i$ for each degree $i$.
\item[\rm (iii)]
    The even  Betti numbers of $\Sat[4]$ satisfy the conditions described below:
   $$  \begin{array}{c|ccccccccccc}
    i&0&2&4&6&8&10&12&14&16&18&20\\\hline
    b_i&1&1&1&3&3&\geq 2& \geq 2& \geq 2& \geq 1& 1&1
  \end{array}
  $$
where all Hodge structures are pure of Tate type with the exception of $\coh[6]{\Sat[4]}=\QQ(-3)^{\oplus 2}+\QQ$ and $\coh[8]{\Sat[4]}=\QQ(-4)^{\oplus 2}+\QQ(-1)$. Furthermore, the odd Betti numbers of $\Sat[4]$ vanish in degree $\leq 7$. 
  \end{itemize}
  \end{thm}
 \begin{proof}
  (i) is \cite[Theorem 1]{huto2} updated by taking into account the result $e(\ab[4])=9$.\\
  (ii) is obtained from altering the spectral sequence of \cite[Table 1]{huto2} in order to compute the rational cohomology of $\Perf[4]$, replacing the first column with the results on the strata of toroidal rank $4$ in \cite[Theorem 25]{huto2}.\\
  (iii) can be proven by considering the Gysin exact sequence in cohomology with compact support associated with the natural stratification of $\Sat[4]$ and taking into account that the cohomology of $\ab[4]$ is known in degree $\geq 12$ 
 and that it cannot contain classes of Hodge weight $8$ by the shape of the spectral sequence of \cite[Table 1]{huto2}.
  \end{proof}

Furthermore, in \cite[Corollary 3]{huto2} we prove that $\coh[12]{\ab[4]}$ is an extension of $\QQ(-6)$ (generated by the tautological class $\lambda_3\lambda_1^3$) by $\QQ(-9)$. A reasonable expectation for 
$\coh{\ab[4]}$ would be that it coincides with $R_3$ in all other degrees.

However, the local system with constant coefficients is not the only one worth while considering for $\ab$. Indeed, looking at cohomology with values in non-trivial 
local systems
has very important applications, both for arithmetic and for geometric questions. Let us recall that $\HH_g$ is contractible and thus the orbifold fundamental group of $\ab$ is isomorphic to $\Sp(2g,\ZZ)$.
Hence representations of $\Sp(2g,\ZZ)$ give rise to (orbifold) local systems on $\ab$.
Recall that
all irreducible representations of $\Sp(2g,\CC)$ are defined over the integers.
 As is well known (\cite[Chapter 17]{fuhabook}, the irreducible representations of  $\Sp(2g,\CC)$ are indexed by their highest weight $\lambda=(\lambda_1,\dots,\lambda_g)$ with 
 $\lambda_1\geq \lambda_2\geq \ldots \geq \lambda_g$. Then for each highest weight $\lambda$ we can define a local system $\VV_\lambda$, as follows. We consider the associated rational representation $\rho_\lambda:\;\Sp(2g,\ZZ)\rightarrow V_\lambda$ and define 
$$
\VV_\lambda:=\Sp(2g,\ZZ) \backslash (\HH_g\times V_\lambda)
\ \ \
M(\tau,v)=(M\tau,\rho_\lambda(M)v).
$$

Each point of the stack $\ab$ has an involution given by the inversion on the corresponding abelian variety. It is easy to check that this involution acts by $(-1)^{w(\lambda)}$ on $\VV_\lambda$, where we call $w(\lambda):=\lambda_1+\cdots+\lambda_g$ the \emph{weight} of the local system $\VV_\lambda$. This implies that the cohomology of all local systems of odd weight is trivial, so that we can concentrate on local systems of even weight.

Let us recall that also cohomology with values in local systems carries mixed Hodge structures. Faltings and Chai studied them in \cite[Chapter VI]{fachbook} using the BGG-complex and found in this way a very explicit description of all possible steps in the Hodge filtration. The bounds on the Hodge weights are those coming naturally from Deligne's Hodge theory:
\begin{thm}
The mixed Hodge structures on the groups $\cohvi[k]{\ab}\lambda$ have weights larger than or equal to $k+w(\lambda)$.
\end{thm}

\begin{rem}
In practice, the formulas for cohomology of local systems are easier to state in terms of cohomology with local support. Let us recall that Poincar\'e duality for $\ab$ implies
$$
\cohcvi\ab\lambda \cong(\cohvi[g(g+1)-\pu]\ab\lambda)^*\otimes \QQ\left(-\frac{g(g+1)}2-w(\lambda)\right)
$$
so that the weights on $\cohcvi[k]\ab\lambda$ are smaller than or equal to $k+w(\lambda)$.
\end{rem}

To understand how cohomology with values in non-trivial local systems behaves. let us consider the case of the moduli space $\ab[1]$ of elliptic curves first. In this case we obtain a sequence of local systems $\VV_k=\Sym^k\VV_1$ for all $k\geq 0$, and $\VV_1$ coincides with the local system $R^1\pi_*{\QQ}$ for $\pi:\;\calX_1\rightarrow\ab[1]$. 
In this case the cohomology of $\ab[1]$ with coefficients in $\VV_k$ is known by Eichler--Shimura theory (see \cite{deligne-modforms}). In particular, by work of Deligne and Elkik \cite{elkik}, the following explicit formula describes the cohomology with compact support of $\ab[1]$:
\begin{equation}
  \label{eq:eichlershimura}
\cohcvi[1]{\ab[1]}{2k}=\bS[2k+2]\oplus \QQ
\end{equation}
where $\bS[2k+2]$ is a pure Hodge structure of Hodge weight $2k+1$ with Hodge decomposition $\bS[2k+2]_\CC=S_{2k+2}\oplus \overline{S_{2k+2}}$ where $S_{2k+2}$ and $\overline{S_{2k+2}}$ 
can be identified with the space of cusp forms of weight $2k+2$ and its complex conjugate respectively.

Formula \eqref{eq:eichlershimura} has been generalized to genus $2$ and $3$ by work of Faber, Van der Geer and Bergstr\"om \cite{fg1,fg2,bfg}. Using the fact that for these values of $g$ the image of the Torelli map is dense in $\ab$, they obtain (partially conjectural) formulas for the Euler characteristic of $\cohcvi\ab\lambda$ in the Grothendieck group of rational Hodge structures from counts of the number of curves of genus $g$ with prescribed configurations of marked points, defined over finite fields. Let us observe that these formulas in general do not give descriptions of the cohomology groups themselves, for instance because cancellations may occur in the computation of the Euler characteristic.

In the case $g=2$ a description of the cohomology groups analogue to that in \eqref{eq:eichlershimura} is  given in \cite{petersen-ab2}. His main theorem is a description of $\cohcvi{\ab[2]}{a,b}$ in terms of cusp forms for $\SL(2,\ZZ)$ and vector-valued Siegel cusp forms for $\Sp(4,\ZZ)$. In particular, for $(a,b)\neq(0,0)$ one has that $\cohcvi[k]{\ab[2]}{a,b}$ vanishes unless we have $2\leq k\leq 4$. The simplest case is the one in which we have $a>b$.
In this case, the result for cohomology with rational coefficients is the following:
\begin{thm} [{\cite[Thm. 2.1]{petersen-ab2}}]
  For $a>b$ and $a+b$ even, we have
\begin{align*}
  \cohcvi[3]{\ab[2]}{a,b}&=\bS_{a-b,b+3}+\bS_{a-b+2}(-b-1)^{\oplus s_{a+b+4}}+\bS_{a+3}
  \\&\phantom{=}
  +\QQ(-b-1)^{\oplus s_{a+b+4}}+\left\{\begin{array}{ll}\QQ(-1) &\text{ if $b=0$,}\\0&\text{ otherwise.}\end{array}\right.
  \\
    \cohcvi[2]{\ab[2]}{a,b}&=\bS_{b+2}+\QQ^{\oplus s_{a-b-2}}+\left\{\begin{array}{ll}\QQ &\text{ if $b>0$ and $a,b$ even,}\\0&\text{ otherwise.}\end{array}\right.,
\end{align*}
Here \begin{itemize}
\item $s_{j,k}$ denotes the dimension of the space of vector-valued Siegel cusp forms for $\Sp(4,\ZZ)$ transforming according to $\Sym^j\otimes \det^{k}$
and $\bS_{j,k}$ denotes a $4s_{j,k}$-dimensional pure Hodge structure of weight $j+2k-3$ with Hodge numbers $h^{j+2k-3,0}=h^{j+k-1,k-2}=h^{k-2,j+k-1}=h^{0,j+2k-3}=s_{j,k}$;
  \item
   $s_k$ denotes the dimension of the space of cusp eigenforms for $\SL(2,\ZZ)$ of weight $k$ and $\bS_{k}$ denotes the corresponding $2s_k$-dimensional weight $k-1$ Hodge structure.
\end{itemize}
   Furthermore $\cohcvi[k]{\ab[2]}{a,b}$ vanishes in all other degrees $k$.
\end{thm}

The formula for the local systems $\VV_{a,a}$ with $a> 0$ is not much more complicated, but it involves also subspaces of cusp eigenforms that satisfy special properties, i.e. not being Saito--Kurokawa lifts, or the vanishing of the central value $L(f,\frac12)$.
Moreover, the group $\cohcvi[2]{\ab[2]}{a,a}$ does not vanish in general but has dimension $s_{2a+4}$.

The proof of the main result of \cite{petersen-ab2} is based on the generally used approach of decomposing the cohomology of $\ab[g]$ into \emph{inner cohomology} and \emph{Eisenstein cohomology}. Inner cohomology is defined as
$$
\cohinvi{\ab}{\lambda}:=\operatorname{Im}[\cohcvi{\ab}\lambda\rightarrow \cohvi{\ab}\lambda],
$$
while Eisenstein cohomology is defined as the kernel of the same map.

The Eisenstein cohomology of $\ab[2]$ with arbitrary coefficients was completely determined by Harder \cite{harder-ab2}.
Thus it is enough to concentrate on inner cohomology.

The \emph{cuspidal} cohomology $\cohcuspvi{\ab}\lambda$  is defined as the image in $\cohvi{\ab}\lambda$ of the space of harmonic $\VV_\lambda$-valued forms whose coefficients are $\VV_\lambda$-valued cusp forms. By \cite[Cor. 5.5]{borel2}, cuspidal cohomology is a subspace of inner cohomology. Since the natural map from cohomology with compact support to cohomology factors through $L^2$-cohomology, we always have a chain of inclusions
$$
\cohcuspvi{\ab}\lambda\subset\cohinloc{\pu}{\ab}{\VV_\lambda\otimes\CC}\subset\cohtwovi{\ab}\lambda.
$$

These inclusions become more explicit if we consider the Hecke algebra action on $\cohtwovi{\ab}\lambda$ that comes from its interpretation in terms of 
$(\mathfrak g,K_\infty)$-cohomology. 
In this way $\cohtwovi{\ab}\lambda$ can be decomposed as direct sum of pieces associated to the elements of the discrete spectrum of $L^2(\Gamma\backslash G)$. Then $\cohcuspvi{\ab}\lambda$ corresponds to the pieces of the decomposition corresponding to cuspidal forms.

Let us recall that one can realize the quotient $\ab[g]/\pm 1$ as a Shimura variety for the group $\operatorname{PGSp}(2g,\ZZ)$. The coarse moduli spaces of $\ab[g]/\pm 1$ and $\ab[g]$ coincide and it is easy to identify local systems on both spaces. Moreover, in the case $g=2$ there is a very precise description of automorphic forms for $\operatorname{PGSp}(4,\ZZ)$
in \cite{flicker}, and in particular of all representations in the discrete spectrum. A careful analysis of these results allows Petersen to prove in \cite[Prop.~4.2]{petersen-ab2} that there is an equality $\cohinvi{\ab[2]}{a,b}=\cohcuspvi{\ab[2]}{a,b}$ 
of the inner and the cuspidal cohomology in this case.
Moreover, the decomposition of the discrete spectrum of $\operatorname{PGSp}(4,\ZZ)$ described by Flicker can be used to obtain an explicit formula for $\cohinvi{\ab[2]}{a,b}$ as well.

We want to conclude this section with a discussion about the intersection cohomology of $\Sat$ and the  toroidal compactifications for small genus. 
For $g\leq 4$ a geometric approach was given in \cite{HGIH} where all intersection Betti numbers with the exception of the middle Betti number in genus $4$ were determined.
As it was pointed out to us by Dan Petersen, representation theoretic methods and in particular the work by O. Ta\"{\i}bi, give an alternative and very powerful method for computing these numbers, as will be 
discussed in the appendix, in particular Theorem \ref{thm:IH_Sat_taut_smallg}.
For ``small'' weights $\lambda$, and in particular for $\lambda=0$ and $g \leq
11$ the classification theorem due to Chenevier and Lannes \cite[Th\'eor\`eme
3.3]{CheLan} is a very effective alternative to compute these intersection Betti
numbers, also using Arthur's multiplicity formula for symplectic groups.
Here we state the following

\begin{thm}\label{teo:Intersectionsmall} 
For $g \leq 5$ there is an isomorphism of graded vector spaces between the intersection cohomology of the Satake compactification $\Sat$ and the tautological ring
\begin{equation*}
IH^{\bullet}(\Sat) \cong R_g.
\end{equation*} 
\end{thm}

\begin{rem}
As explained in the appendix, see Remark \ref{rem:ihsatcomments}, this result is sharp. For $g \geq 6$ there is a proper inclusion $R_g \subsetneq IH^{\bullet}(\Sat)$ and starting from $g=9$ there is even
non-trivial intersection cohomology in odd degree. See also Example
\ref{exa:IH_as_spin} for the computation of intersection cohomology for $g=6,7$.
\end{rem}

\begin{proof}
There are two possibilities to compute the intersection cohomology of $\Sat$, at least in principle, explicitly. The first one is {\em geometric} in nature, the second uses {\em representation theory}. we shall first
discuss the geometric approach. This was developed in \cite{HGIH} and is based  on the decomposition theorem due to Beilinson, Bernstein, Deligne and Gabber. For this we refer the reader to the 
excellent survey paper by de Cataldo and Migliorini \cite{deCataldoMigliorini}.
 We shall discuss this here in the special case of genus $4$. 
We use the stratification of the Satake compactification given by  
\begin{equation*}\label{eq:strata-satake4}
\Sat[4] = \ab[4] \sqcup \ab[3] \sqcup \ab[2] \sqcup \ab[1]  \sqcup \ab[0].
\end{equation*}
In this genus the morphism 
$\varphi: \Vor[4] \to \Sat[4]$ is a resolution of singularities (up to finite quotients). We denote
\begin{equation*}
\beta_i^0:= \varphi^{-1}(\ab[4-i]).
\end{equation*}
Then $\varphi|_{\beta_i^0}: \beta_i^0 \to \ab[4-i]$
is a topological fibration (but the fibres are typically not smooth). This is the basic set-up of the decomposition theorem.  
Since $\Vor[4]$ is rationally smooth, its cohomology and intersection cohomology coincide. Taking into account that the complex dimension of $\Vor[4]$ is $10$ and that
$\ab[k]$ has dimension $k(k+1)/2$, the
decomposition theorem then gives the following (non canonical) isomorphism 
\begin{equation}\label{equ:inclusion}
H^m(\Vor[4], \QQ)  \cong  IH^m(\Sat[4],\QQ) \oplus  \bigoplus_{k <4, i, \beta}^{} IH^{m- 10 + k(k+1)/2 + i} (\Sat[k], {\mathcal L}_{i,k,\beta}).
\end{equation}
where the ${\mathcal L}_{i,k,\beta}$ are certain local systems on $\ab[k]$ and the integer $i$ runs through the interval $[-r(\varphi), r(\varphi)]$ where $r(\varphi)$ is the defect of the map $\varphi$.
The basic idea is this: if one can compute the cohomology on the right hand side of this formula and compute the local systems ${\mathcal L}_{i,k,\beta}$ and their cohomology, then 
one can also compute the intersection cohomology one is interested in. The advantage is that $\Vor$ is a rationally smooth space whose geometry one knows well and that the 
local systems  ${\mathcal L}_{i,k,\beta}$  live on a smaller dimensional space. 

For $k < 4$ we have inclusions $\ab[k] \hookrightarrow \Sat[4]$. Taking a transverse slice at a point $x\in \ab[k]$ in $\Sat[4]$ 
defines the link $N_{k,x}$
and varying the point $x$ gives us the link bundle 
${\mathcal N}_{k}$. Then, taking the intersection cohomology of the  links  ${N}_{k,x}$ we obtain local systems $\mathcal I \mathcal H^j(\mathcal N_k)$ on $\ab[k]$.
Not all of these local systems will play a role in the decomposition theorem, but one truncates this collection of local systems and only considers  those  in the range $j < \codim \beta_k^0$. 
For the biggest non-trivial stratum this means the following: 
the complex codimension of $\ab[3]$ in $\Vor[4]$ is $10-6=4$, and hence we have to consider the local systems  $\mathcal I \mathcal H^j(\mathcal N_3)$  for $j\leq 3$.
In fact these can be determined explicitly and one finds that  $\mathcal I \mathcal H^0(\mathcal N_3)=\QQ$,  $\mathcal I \mathcal H^j(\mathcal N_3)=0$ for $j=1,3$ and 
 $\mathcal I \mathcal H^2(\mathcal N_3)= \VV_{11}$.  These local systems carry suitable weights, but as these play no essential role for the proof which we will outline below, we shall omit the Tate twists here.   

We then have to compare this to the direct image $R\varphi_*(\QQ)$.
For this recall that $\varphi|_{\beta_{4-k_1}^0}: \beta_{4-k_1}^0 \to \ab[k_1]$ is a topological fibre bundle.
To explain the mechanism of the decomposition theorem we again consider the first non-trivial stratum, namely $k=3$, i.e. the biggest stratum.
${\mathcal H}^j(\beta_1^0):=  (R^j{\varphi|_{\beta_{1}^0}})_*(\QQ)$. All local systems $\mathcal I \mathcal H^j(\mathcal N_3)$ in the non-truncated range are direct summands 
of  ${\mathcal H}^j(\beta_1^0):=  (R^j{\varphi|_{\beta_{1}^0}})_*(\QQ)$. However, there will be further local systems, which are not accounted for by the truncated link cohomology,
and their cohomology will appear as summands on the right hand side of the decomposition theorem (\ref{equ:inclusion}). In our case 
 ${\mathcal H}^0(\beta_1^0)= \QQ$ and  ${\mathcal H}^2(\beta_1^0)= \VV_{11}$ are accounted for by the link cohomology, but the 
local systems   ${\mathcal H}^4(\beta_1^0)= \VV_{11}$ and 
 ${\mathcal H}^6(\beta_1^0)=\QQ$ will contribute to the right hand side of  (\ref{equ:inclusion}). All local systems come with a certain weight (Tate twist), but we do not discuss this here.
 In this situation the computation of the local systems  ${\mathcal H}^j(\beta_k^0)$ is easy as $\beta_1^0 \to \ab[3]$ is the universal Kummer family.
 In general this is much harder, but the necessary computations were done in \cite{huto2}.

We then have to continue this process for the smaller dimensional strata and in each case determine the local systems which are not accounted for by the truncated 
intersection cohomology of the link bundles and the locals systems arising from bigger strata. This calculation works surprisingly well and even provides information on the vanishing of the intersection
cohomology of certain links and local systems. The  missing piece of information in \cite{HGIH} was  the intersection cohomology group 
$IH^6(\Sat[3], \VV_{11})$ which indeed  vanishes, see Corollary (\ref{cor:IHSat3V11}) in the appendix.   

This result can also be proven completely by representation theoretic methods. For this we refer the reader to Theorem (\ref{thm:IH_Sat_taut_smallg}) and its proof.
\end{proof}

Finally we turn to  toroidal compactifications. As we have already explained all standard toroidal compactifications coincide in genus $\leq 3$ 
 and are smooth up to finite
quotient singularities. In particular, singular  cohomology and  intersection cohomology coincide.  In genus $g=4$ the Voronoi compactification $\Vor[4]$ is still rationally smooth.
In this genus the perfect  cone compactification and the central cone (or Igusa) compactification coincide: $\Perf[4]\cong\Centr[4]$. Hence the only interesting space from the 
point of view of intersection cohomology is $\Perf[4]$. In this case we also have a morphism $\Vor[4] \to \Perf[4]$ which is, again up to taking quotients with respect to  finite group,
a resolution of singularities.   
\begin{thm}\label{teo:Intersectionperf} 
All the odd degree intersection Betti numbers of $\Perf[4]$ are zero, while the even ones
$ib_j:=\dim IH^j(\Perf[4]),\QQ$ are as follows:
\begin{equation} \label{equ:bettiperf4}
\begin{array}{r|ccccccccccc}
j&0&2&4&6&8&10&12&14&16&18&20\\\hline
ib_j&1&2&4&9&14&16&14&9&4&2&1
\end{array}
\end{equation}
\end{thm}
\begin{proof}
This is a straightforward application of the decomposition theorem to the resolution $\Vor[4] \to \Perf[4]$. The result then follows from Theorem \ref{teo:cohomologyA4}
together with the knowledge of the 
cohomology of the exceptional divisor $E$, which was computed in \cite[Theorem 26 (1)]{huto2}. 
\end{proof}

\section{Stabilization results}\label{sec:stabilization}

As we have seen in the previous section, obtaining complete results on the cohomology of $\ab$ and its compactifications is very hard even for small  $g$.
In general, fixing an integer $g$ and wanting to know the cohomology in all degrees $k$ is hopeless.
The situation is very different if one instead fixes a degree $k$ and studies the dependency of the $k$th cohomology group on $g$. The reason for this lies in the fact that there are natural maps relating $\ab$ to $\ab[g']$ for $g\leq g'$.

Let us recall that the product of two principally polarized abelian varieties has naturally the structure of a principally polarized abelian variety. At the level of moduli spaces, this implies the existence of natural product maps
$$\begin{array}{rccc}
  \pr:&\ab[g_1]\times \ab[g_2]& \longrightarrow& \ab[g_1+g_2]\\
  & (A,B) &\longmapsto & A\times B.
  \end{array}
$$

In particular, if one fixes an elliptic curve $E\in\ab[1]$ one obtains a sequence of maps $\ab \rightarrow \ab[g+1]$ which induce natural maps $s_g:\; \coh{\ab}\rightarrow \coh{\ab[g+1]}$. As all points in $\ab[1]$ are homotopy equivalent, the maps $s_g$ do not depend on the choice of $E$. The maps $s_g$ can also be described independently by identifying $\coh{\ab}$ with the cohomology of the symplectic group $\Sp(2g,\ZZ)$ and constructing them as the maps in cohomology associated with the inclusion 
$$\begin{array}{ccc}
  \Sp(2g,\ZZ)& \lhook\joinrel\longrightarrow  &\Sp(2g+2,\ZZ)\\
  \left(
  \begin{array}{c|c}
    A&B\\\hline
    C&D
    \end{array}
  \right)
  &\longmapsto &
\left(
  \begin{array}{cc|cc}
    A&\underline{\mathbf 0}^t&B&\underline{\mathbf 0}^t\\
    \underline{\mathbf 0}&1&\underline{\mathbf 0}&0\\\hline
        C&\underline{\mathbf 0}^t&D&\underline{\mathbf 0}^t\\
    \underline{\mathbf 0}&0&\underline{\mathbf 0}&1\\
  \end{array}
 \right).
  \end{array}$$

It is natural to ask
\begin{qu}\label{q:stab}
  Is the map $s_g:\;\coh[k]{\ab[g+1]}\rightarrow \coh[k]{\ab}$ an isomorphism for $g\gg k$?
\end{qu}
If the answer to this question is affirmative, then we say that the cohomology of $\ab$ {\em stabilizes} with respect to $g$; the range of values of $k$ for which $s_g$ is an isomorphism is called the 
{\em stability range}. In the stable range, cohomology coincides with the inductive limit
$$
\coh[k]{\ab[\infty]}:=\varprojlim_g{\coh[k]{\ab}}
$$
which has a natural structure as a Hopf algebra, with products defined by the usual cup products and coproducts defined by the pull-backs of the product maps $\pr$. In particular, by Hopf's theorem (see e.g. \cite[Theorem 3C.4]{hatcher-book} stable cohomology, when it exists, is always a freely generated graded-commutative algebra.

By now, cohomological stability is known for large classes of arithmetic groups. The case of $\Sp(2g)$ (and hence $\ab$) is a classical result due to Borel. It was first proved in \cite{borel1} by identifying the cohomology of $\ab$ in the stable range with the space of $\Sp(2g,\ZZ)$-invariant harmonic forms of the appropriate degree.
The key point is that in the stable range cohomology can be calculated using differential forms that satisfy a certain logarithmic growth condition near the boundary. This part of the construction relies on sheaf theory on the Borel--Serre compactification $\BS$ of \cite{BorelSerre}.
Later in \cite{borel2}, Borel realized that it was possible to extend his stability theorem also to arbitrary local systems on $\ab$ and that a more flexible choice of the growth condition considered would give a better stability range. In the case of the arithmetic group $\Sp(2g,\ZZ)$, this range was made explicit in \cite[Thm.~3.2]{hain-infinitesimal}. The complete result for $\ab$ can be summarized as follows:

\begin{thm}[\cite{borel1,borel2}]
\label{thm:stable}
The cohomology of $\ab$ with $\QQ$-coefficients stabilizes in degree $k<g$ and in this range it is freely generated by the odd $\lambda$-classes $\lambda_1,\lambda_3,\lambda_5,\ldots$. 
Furthermore, the cohomology of $\ab$ with values in an arbitrary local system $\VV$ vanishes in degree $k<g$ unless $\VV$ has a non-trivial constant summand.
\end{thm}

The classical result above gives the best known bound for the stability range, with one important exception, that of so called \emph{regular} local systems, i.e. the $\VV_{\lambda}$ with indices $\lambda_i$ all distinct and positive. Namely, according to Saper's theory of $\mathcal{L}$-sheaves on the reductive Borel--Serre compactification, we have
\begin{thm}[{\cite{saper1}}]
If $\lambda_1>\dots>\lambda_g>0$, then $\cohvi[k]{\ab}{\lambda_1,\dots,\lambda_g}$ vanishes for $k<\dim_\CC\ab$.
\end{thm}
The theorem above is a special case of \cite[Theorem 5]{saper1}, which holds for all quotients of a Hermitian symmetric domain or equal-rank symmetric space. Expectedly, Saper's techniques can be employed to give better bounds for the vanishing of certain classes of non-trivial local systems as well.

The fact that cohomology with values in non-trivial local systems vanishes implies that in small degree, it is easy to describe the cohomology of spaces that are fibered over $\ab$. In particular, Theorem~\ref{thm:stable} implies that the cohomology of the universal family $\calX_g\rightarrow \ab$ and that of its fiber products stabilize:
\begin{thm}[{\cite[Thm. 6.1]{grhuto}}]\label{thm:stableuniv}
  For all $n$, the rational cohomology of the $n$th fibre product $\calX_g^{\times n}$ of the universal family stabilizes in degree $k<g$ and in this range it is isomorphic to the free $\coh{\ab[\infty]}$-algebra generated by the classes $T_i:= p_i^*(\Theta)$ and $P_{jk}:=p_{jk}^*(P)$ for $i=1,\dots,n$ and $1\leq j<k\leq n$, where
  $$p_i:\;\calX_g^{\times n}\rightarrow \calX_g,
  \ \ \ 
  p_{jk}:\;\calX_g^{\times n}\rightarrow \calX_g^{\times 2}$$
  are the projections and 
  we denote by $\Theta\in \coh[2]{\calX_g}$ and $P\in \coh[2]{\calX_g^{\times 2}}$  the class of the universal theta divisor and of the universal Poincar\'e divisor normalized along the zero section, respectively.
\end{thm}

The next natural question is whether cohomological stability also holds  for compactifications of $\ab$. This is of particular relevance useful for families of compactifications to which the product maps $\pr:\;\ab[g_1]\times\ab[g_2]\rightarrow \ab[g_1+g_2]$ extend. The analogue of Question~\ref{q:stab} for the Baily--Borel--Satake compactification was settled already in the 1980's by Charney and Lee:
\begin{thm}[{\cite{chle}}]
  The rational cohomology of $\Sat$ stabilizes in degree $k<g$. in this range, the cohomology is isomorphic to the polynomial algebra
  $$
  \QQ[x_2,x_6,\dots x_{4i+2},\dots] \otimes \QQ[y_6,y_{10},\dots y_{4j+2},\dots]
  $$
  generated by classes $x_{4i+2}$ ($i\geq 0$) and $y_{4j+2}$ ($j\geq 1$) of degree $4i+2$ and $4j+2$ respectively.
\end{thm}

It follows from Charney and Lee's construction that the classes $x_{4i+2}$ restrict to $\lambda_{2i+1}$ on $\ab$, whereas the classes $y_{4j+2}$  vanish on $\ab$.

The proof of the theorem above combines Borel's results about the stable cohomology of $\Sp(2g,\ZZ)$ and $\GL(n,\ZZ)$ with techniques from homotopy theory. First, it is proved that the rational cohomology of $\Sat$
is canonically isomorphic to the cohomology of the geometric realization of Giffen's category $W_g$, which arises from Hermitian $K$-theory. The limit for $g\geq \infty$ of these geometric realizations $\left|W_g \right|$ can then be realized as the base space of a fibration whose total space has rational cohomology isomorphic to $\coh{\Sp(\infty,\ZZ)}$ and whose fibres have rational cohomology isomorphic to $\coh{\GL(\infty,\ZZ)}$. This immediately yields the description of the generators of the stable cohomology of $\Sat$. The stability range is proved by looking directly 
at the stability range for Giffen's category. 
In particular, this part of the proof is independent of Borel's constructions and shows that the stability range for the cohomology of $\ab$ should indeed be $k<g$.

However, because of the fact that Charney and Lee replace $\Sat$ with its $\QQ$-homology equivalent space $\left|W_{g}\right|$, the geometric meaning of the $x$- and $y$-classes remains unclear. This gives rise to the following two questions:

\begin{qu}\label{qu:eduard1}
What is the geometrical meaning of the classes $y_{4j+2}$? In particular, what is their Hodge weight?
\end{qu}
\begin{qu}\label{qu:eduard2}
Is there a canonical way to lift $\lambda_{4i+2}$ from $\coh{\ab}$ to $\coh{\Sat}$ for $4i+2<g$?
\end{qu}

The answer to the first question was obtained recently by Chen and Looijenga in \cite{chlo}. Basically, in their paper they succeed in redoing Charney--Lee's proof using only algebro-geometric constructions. 
In particular, they work directly on $\Sat$ rather than passing to Giffen's category and study its rational cohomology by investigating the Leray spectral sequence associated with 
the inclusion $\ab\hookrightarrow\Sat$. The $E_2$-term of this spectral sequence can be described explicitly using the fact that each point in a stratum $\ab[k]\subset \Sat$ has an 
arbitrarily small neighbourhood which is a virtual classifying space for an arithmetic subgroup $P_g(k)\subset\Sp(2g,\ZZ)$ which is fibered over $\GL(g-k,\ZZ)$. In the stable range
this can be used to construct a spectral sequence converging to $\coh{\ab}$ with $E_2$-terms isomorphic to those in the spectral sequence considered by Charney and Lee. 
Furthermore, this algebro-geometric approach allows to describe explicitly the Hodge type of the $y$-classes. This is done by first giving a ``local'' interpretation of them as 
classes lying over the cusp $\ab[0]$ of $\Sat$, and then using the existence of a toroidal compactification to describe the Hodge type of the $y$-classes in the spirit of 
Deligne's Hodge theory \cite{deligne-hodgeIII}. This gives the following result:

\begin{thm}[{\cite[Theorem 1.2]{chlo}}] The $y$-classes have Hodge type $(0,0)$.
  \label{thm:stabsat}
\end{thm}

Concerning Question~\ref{qu:eduard2}, it is clear that while the $y$-classes are canonically defined, the $x$-classes are not. On the other hand, Goresky and Pardon \cite{GoreskyPardon}
defined \emph{canonical} lifts of the $\lambda$-classes from $\lambda_i\in\cohloc{\pu}{\ab}\CC$ to $\lambda^{\op{GP}}_i\in\cohloc{\pu}{\Sat}\CC$. As the pull-back of $x_{4i+2}$ and of $\lambda^{\op{GP}}_{2i+1}$ to any smooth toroidal compactification of $\ab$ coincide, one wonders whether the two classes may coincide. This question was settled in the negative by Looijenga in \cite{lo-pardon}, who studied the properties of the Goresky--Pardon classes putting them in the context of the theory of stratified spaces. We summarize the main results about stable cohomology from \cite{lo-pardon} as follows
\begin{thm}[\cite{lo-pardon}]
  \label{thm:lo-pardon}
  For $2<4r+2<g$, the Goresky--Pardon lift of the degree $2r+1$ Chern character of the Hodge bundle has a non-trivial imaginary part and its real class lies in $\coh[4r+2]{\Sat}$. In particular, the classes $\lambda^{\op{GP}}_{2r+1}$ are different from the $x_{4r+2}$ in Theorem~\ref{thm:stabsat}.
\end{thm}

This is related to an explicit description of the Hodge structures on a certain subspace of stable cohomology.
Let us recall that an element $x$ of a Hopf algebra is called \emph{primitive} if its coproduct satisfies $\Delta(x)=x\otimes 1+1\otimes x$.
If one considers the Hopf algebra structure of stable cohomology of $\Sat$, the primitive part in degree $4r+2$ is generated by $y_{4r+2}$ and the Goresky--Pardon lift $\op{ch}^{\op{GP}}_{2r+1}$ of the Chern character, which is a degree $2r+1$ polynomial in the $\lambda_{j}^{\op{GP}}$ with $j\leq 2r+1$.
The proof of Theorem~\ref{thm:lo-pardon} is based on an explicit computation of the Hodge structures on this primitive part (\cite[Theorem 5.1]{lo-pardon}) obtained by 
using the theory of Beilinson regulator and the explicit description of the $y$-generators given in \cite{chlo}.
This amounts to describing ${\coh[4r+2]{\Sat}}_{\op{prim}}$ as an extension of a weight $4r+2$ Hodge structure generated by $\op{ch}_{2r+1}^{\op{GP}}$ by a weight $0$ Hodge structure, generated by
$y_{4r+2}$. 
Chen and Looijenga's explicit construction of $y_{2r+2}$ also yields a construction of a \emph{homology} class $z\in H_{4r+2}(\Sat,\QQ)$ that pairs non-trivially with $y_{4r+2}$. Hence, one can describe the class of the extension by computing the pairing of $\op{ch}_{2r+1}^{\op{GP}}$ with $z$. This computation is only up to multiplication by a non-zero rational number because of an ambiguity in the definition of $z$, but it is enough to show that the class of the extension is real and a non-zero rational multiple of $\frac{\zeta(2r+1)}{\pi^{2r+1}}$.

Let us mention that the question about stabilization is settled also for the reductive Borel--Serre compactification  $\ab^{\operatorname{RBS}}$ of $\ab$. Let us recall from Section~\ref{s:Ltwocohomology} that its cohomology is naturally isomorphic to the $L^2$-cohomology of $\ab$ and to  the intersection homology of $\Sat$. Combining this with Borel's Theorems \ref{thm:stable} and~\ref{teo:L2IHSat} we can obtain
\begin{thm} \label{thm:stable_IH}
  The intersection cohomology $\icoh[k]{\Sat}$ stabilizes in degree $k<g$ to the graded vector space $\QQ[\lambda_1,\lambda_3,\dots]$. 
  \end{thm}
At this point we would  like to point out that the range in which intersection cohomology is tautological given in this theorem
can be improved considerably, namely to a wider range $k < 2g-2$, and also extended to non-trivial local systems $\VV_{\lambda}$, for which the 
intersection cohomology vanishes in the stable range. 
For details we refer to Theorem \ref{thm:sharper_stab_IH_Sat_level_one} in the appendix.

The analogue of Question~\ref{q:stab} for toroidal compactifications turns out to be a subtle question, which in this form remains open. 
We dealt with stability questions in a series of papers \cite{grhuto,grhuto2}, joint with Sam Grushevsky.

Let us recall that toroidal compactifications come in different flavours. The first question to answer is which choice of toroidal compactification is suitable in order to obtain stabilization phenomena in cohomology. At a theoretical level, this requires to work with a sequence of compactifications $\{\ab^{\Sigma_g}\}$ where each $\Sigma_g$ is an admissible fan in $\Sym^2_{\op{rc}}(\RR^g)$. Then the system of maps $\ab\rightarrow \ab[g+1]$ extends to the compactification if and only if $\{\Sigma_g\}$ is what is known as an \emph{admissible} collection of fans (see e.g. \cite[Def. 8]{grhuto2}). If one wants to ensure that the product maps $\ab[g_1]\times\ab[g_2]\rightarrow \ab[g_1+g_2]$ extend, one needs a stronger condition, which we shall call \emph{additivity}, namely that the direct sum of a cone $\sigma_1\in\Sigma_{g_1}$ and a cone $\sigma_2\in\Sigma_{g_2}$ should always be a cone  
in $\Sigma_{g_1+g_2}$.

All toroidal compactifications we mentioned so far are additive. However, only the perfect cone compactification is clearly a good candidate for stability. For instance, the second Voronoi decomposition is ruled out because the number of boundary components of $\Vor$ increases with $g$, so that the same should happen with $\coh[2]{\Vor}$. Instead, the perfect cone compactification has the property that for all $k\leq g$, the preimage of $\ab[g-k]\subset \Sat$ in $\Perf$ always has codimension $k$. In particular, the boundary of $\Perf$ is always irreducible.

Let us recall that a toroidal compactification associated to a fan $\Sym^2_{\rc}(\RR^g)$ is the disjoint union of locally closed strata $\stratum\sigma$ corresponding to the cones $\sigma\in\Sigma$ up to $\GL(g,\ZZ)$-equivalence.
Each stratum $\stratum\sigma$ has codimension equal to $\dim_\RR\sigma$
by which we mean the dimension of the linear space spanned by $\sigma$.
The {\em rank} of a cone $\sigma$ is defined as the minimal $k$ such that $\sigma$ is a cone in $\Sigma_k$.
If the rank is $k$,
then $\stratum\sigma$ maps surjectively to $\ab[g-k]$ under the forgetful morphism to $\Sat$.

The properties of the perfect cone decomposition  can be rephrased in terms of the fan by saying that if a cone $\sigma$ in the perfect cone decomposition has rank $k$, then its dimension is at least $k$. Moreover, the number of distinct $\GL(g,\ZZ)$-orbits of cones of a fixed dimension $\ell\leq g$ is independent of g. This means that the combinatorics of the strata $\stratum\sigma$ of codimension $\ell \leq k$ is independent of $g$ provided $g\geq k$ holds.
Furthermore,
studying the Leray spectral sequence associated to the fibration $\stratum\sigma\rightarrow\ab[g-r]$ with $r=\rank \sigma$ allows to prove that the cohomology of $\stratum\sigma$ stabilizes for $k<g-r-1$; this stable cohomology consists of algebraic classes and can be described explicitly in terms of the geometry of the cone $\sigma$ (see \cite[Theorem~8.1]{grhuto}).
The basic idea of the proof is analogous to the one used in Theorem~\ref{thm:stableuniv} to describe the cohomology of $\calX_g^n$.
All this suggests that $\Perf$ is a good candidate for stability. In practice, however, the situation is complicated by the singularities of $\Perf$.

Let us review the main results of \cite{grhuto,grhuto2} in the case of an arbitrary sequence $\{\Asigma\}$ of partial toroidal compactifications of $\ab$ associated with an admissible collection of (partial) fans $\ssigma=\{\Sigma_g\}$. Here we use the word \emph{partial} to stress the fact that we don't require the union of all $\sigma\in\Sigma_g$ to be equal to $\Sym^2_{\rc}(\RR^g)$. In other words, we are also considering the case in which $\Asigma$ is the union of toroidal open subsets of a (larger) toroidal compactification.
\begin{thm}\label{thm:simplicialcase}
  Assume that $\ssigma$ is an additive collection of fans and that each cone $\sigma\in\Sigma_g$ of rank at least $2$ satisfies
 \begin{equation} \label{eq:dimrank}
  \dim_\RR\sigma\geq \frac{\rank \sigma}2+1.
\end{equation}
Then if $\ssigma$ is simplicial we have that $\coh[k]{\Asigma}$ stabilizes for $k<g$ and that cohomology is algebraic in this range, in the sense that cohomology coincides with the image of the Chow ring.
Furthermore, there is an isomorphism 
$$
\coh[2\pu]{\Asigma[\infty]} \cong
\coh[2\pu]{\ab[\infty]} \otimes_\QQ \Sym^\pu(V_\ssigma)
$$
of free graded algebras between stable cohomology and the algebra over 
$$\coh[2\pu]{\ab[\infty]}=\QQ[\lambda_1,\lambda_3,\dots]$$
generated by the symmetric algebra of the graded vector space
spanned by the tensor products of forms in the $\QQ$-span of each cone $\sigma$
that are invariant under the action of the stabilizer $\Aut\sigma$ of $\sigma$,
for all cones $\sigma$ that are irreducible with respect to the operation of taking direct sums, i.e. 
$$
V_\ssigma^{2k}:=\bigoplus_{\substack{[\sigma]\in[\ssigma]\\ [\sigma]\text{ irreducible w.r.t.}\ \oplus}} \left( \Sym^{k-\dim_\RR\sigma}(\QQ\text{-span of }\sigma)\right)^{\Aut\sigma}
$$
where $[\ssigma]$ denotes the collection of the
orbits of cones in $\ssigma$ under the action of the general linear group.
\end{thm}

As we already remarked, the perfect cone compactification satisfies a condition which is stronger than \eqref{eq:dimrank}, so any rationally smooth open subset of $\Perf$ which is defined by an additive fan satisfies the assumptions of the theorem. For instance, this implies that the theorem above applies to the case where $\Asigma$ is the smooth locus of $\Perf$ or the locus where $\Perf$ is rationally smooth.
A more interesting case that satisfies the assumptions of the theorem is the matroidal partial compactification $\Matr$ defined by the fan $\ssigma^{\op{Matr}}$ of cones defined starting from simple regular matroids. This partial compactification was investigated by Melo and Viviani~\cite{mevi} who showed that the matroidal fan with coincides the intersection $\Sigma_g^{\op{Matr}}=\Sigma_g^{\op{Perf}}\cap\Sigma_g^{\op{Vor}}$  of the perfect cone and second Voronoi fans, so that $\Matr$ is an open subset in both $\Perf$ and $\Vor$. Its geometrical significance is also related to the fact that the image of the extension of the Torelli map to the Deligne--Mumford stable curves is contained in $\Matr$, as shown by Alexeev and Brunyate in \cite{albr}.

Furthermore, by \cite[Prop. 19]{grhuto2}, rational cohomology stabilizes also without the assumption that $\ssigma$ be additive. In this case stable cohomology is not necessarily a free polynomial algebra, but it still possesses an explicit combinatorial description as a graded vector space.

If $\ssigma$ is not necessarily simplicial, it is not known whether cohomology stabilizes in small degree. However, one can prove that cohomology (and homology) stabilize in \emph{small codegree}, i.e. close to the top degree $g(g+1)$. The most natural way to phrase this is to look at Borel--Moore homology of $\Asigma$, because this is where the cycle map from the Chow ring of $\Asigma$ takes values if $\Asigma$ is singular and possibly non-compact. (If $\Asigma$ is compact, then Borel--Moore homology coincides with the usual homology).

If $\ssigma$ is additive, it is possible to prove (see \cite[Prop. 9]{grhuto2}) that the product maps extend, after going to a suitable level structure, to a transverse embedding $\Asigma[g_1]\times\Asigma[g_2]\rightarrow\Asigma[g_1+g_2]$. In particular, taking products with a chosen point $E\in\ab[1]$ defines a transverse embedding (in the stacky sense) of $\Asigma$ in $\Asigma[g+1]$. It makes sense to wonder in which range the Gysin maps
$$
\BM[(g+1)(g+2)-k]{\Asigma[g+1]}\rightarrow\BM[g(g+1)-k]{\Asigma[g]}
$$
are isomorphisms. Then Theorem~\ref{thm:simplicialcase} generalizes to

\begin{thm}\label{thm:singularcase}
  If $\ssigma$ is an additive collection of (partial) fans such that each cone of rank $\geq 2$ satisfies \eqref{eq:dimrank}, then the Borel--Moore homology of $\Asigma$ stabilizes in codegree $k<g$ and the stable homology classes lie in the image of the cycle map.

  Furthermore, there is an isomorphism of graded vector spaces
$$
V_\ssigma^{2k}:=\bigoplus_{\substack{[\sigma]\in[\ssigma]\\ [\sigma]\text{ irreducible w.r.t.}\ \oplus}} \left( \Sym^{k-\dim_\RR\sigma}(\QQ\text{-span of }\sigma)\right)^{\Aut\sigma}.
$$
\end{thm}  

Let us remark that Borel--Moore homology in small codegree does not have a ring structure \emph{a priori}.

As the assumptions of the theorem above are satisfied by the perfect cone compactification, we get the following result.
\begin{cor}[{\cite[Theorems 1.1\& 1.2]{grhuto}}]
  The rational homology and cohomology of $\Perf$ stabilizes in small codegree, i.e. in degree $g(g+1)-k$ with $k<g$.
  In this range, homology is generated by algebraic classes. 
\end{cor}

As explained in Dutour-Sikiri\'c  appendix to \cite{grhuto2}, the state of the art of the classification of orbits of matroidal cones and perfect cone cones is enough to be able to compute the stable Betti numbers of $\Matr$ in degree up to $30$ and the stable Betti numbers of $\Perf$ in codegree at most $22$ (where the result for codegree $22$ is actually a lower bound, see \cite[Theorems 4 \& 5]{grhuto2}).

Concluding, we state two problems that remain open at the moment.
\begin{qu}\label{qu:stableperf}
  Does the cohomology of $\Perf$ stabilize in small degree $k<g$?
\end{qu}
As we have already observed, the answer to Question~\ref{qu:stableperf} is related to the behaviour of the singularities of $\Perf$ and a better understanding of them may be necessary to answer this question.

A different question arises as follows.
Since stable homology of $\Perf$ is algebraic, stable homology classes of degree $g(g+1)-k$ can be lifted (non-canonically) to intersection homology of the same degree. A natural question to ask is whether all intersection cohomology classes in degree $k<g$ are of this form, which we may rephrase using Poincar\'e duality as 
\begin{qu}
Is there an isomorphism $H_{g(g+1)-k}({\Perf},\QQ)\cong\icoh[k]{\Perf}$
  for all $k<g$?
\end{qu}

%\appendix
\vfil\eject

\section*{Appendix. Computation of intersection cohomology using the Langlands program}
\label{secappOT}

In this appendix we explain a method for the explicit computation of the Euler
characteristics for both the cohomology of local systems $\VV_{\lambda}$ on
$\ab$ and their intermediate extensions to $\Sat$. Furthermore, we explain how
to compute individual intersection cohomology groups in the latter case. The
main tools here are trace formulas and results on automorphic representations,
notably Arthur's endoscopic classification of automorphic representations of
symplectic groups \cite{Arthur_book}.

We start by explaining in Proposition \ref{prop:eA4} the direct computation of
$e(\ab[4])=9$. This Euler characteristic, as well as $e(\ab, \VV_{\lambda})$ for
$g \leq 7$ and any $\lambda$, can be obtained as a byproduct of computations
explained in the first part of \cite{Taibi_dimtrace}, which focused on
$L^2$-cohomology. In fact by Proposition~\ref{prop:simplification_Tell_sc} these
are given by the conceptually simple formula \eqref{equ:Tell}. The difficulty in
evaluating this formula resides in computing certain coefficients, called
\emph{masses}, for which we gave an algorithm in \cite{Taibi_dimtrace}. The
number $e(\ab[4])$ was missing in \cite{huto2} to complete the proof of Theorem
\ref{teo:cohomologyA4}.

Next we recall from \cite{Taibi_dimtrace} that the automorphic representations
for $\Sp_{2g}$ contributing to $IH^{\bullet}(\Sat,\VV_{\lambda})$ can be
reconstructed from certain sets of automorphic representations of general linear
groups, which we shall introduce in Definition \ref{def:three_sets}. Thanks to
Arthur's endoscopic classification \cite{Arthur_book} specialized to level one
and the identification by Arancibia, Moeglin and Renard \cite{AMR} of certain
Arthur-Langlands packets with the concrete packets previously constructed by
Adams and Johnson \cite{AdJo} in the case of the symplectic groups, combined
with analogous computations for certain special orthogonal groups, we have
computed the cardinalities of these ``building blocks''. Again, this is explicit
for $g \leq 7$ and arbitrary $\lambda$. For $g \leq 11$ and ``small'' $\lambda$,
the classification by Chenevier and Lannes in \cite{CheLan} of level one
algebraic automorphic representations of general linear groups over $\QQ$ having
``motivic weight'' $\leq 22$ (see Theorem \ref{thm:small_weight} below) gives
another method to compute these sets. Using either method, we deduce
$IH^6(\Sat[3], \VV_{1,1,0}) = 0$ in Corollary \ref{cor:IHSat3V11}, which was a
missing ingredient to complete the computation in \cite{HGIH} of
$IH^{\bullet}(\Sat[4], \QQ)$ (case $g=4$ in Theorem
\ref{teo:Intersectionsmall}). In fact using the computation by Vogan and
Zuckerman \cite{VoZu} of the $(\mathfrak{g}, K)$-cohomology of Adams-Johnson
representations, including the trivial representation of $\Sp_{2g}(\RR)$, we can
prove that the intersection cohomology of $\Sat$ is isomorphic to the
tautological ring $R_g$  for all $g \leq 5$ (see Theorem
\ref{teo:Intersectionsmall}), again by either method.

One could deduce from \cite{VoZu} an algorithm to compute intersection
cohomology also in the cases where there are non-trivial representations of
$\Sp_{2g}(\RR)$ contributing to $IH^{\bullet}(\Sat, \VV_{\lambda})$, e.g.\ for
all $g \geq 6$ and $V_{\lambda} = \QQ$. Instead of pursuing this, in Section
\ref{sec:app_Shimura} we make explicit the beautiful description by Langlands
and Arthur of $L^2$-cohomology in terms of the Archimedean Arthur-Langlands
parameters involved, i.e.\ Adams-Johnson parameters. In fact the correct way to
state this description would be to use the endoscopic classification of
automorphic representations for $\GSp_{2g}$. Although this classification is not
yet known, we can give an unconditional recipe in the case of level one
automorphic representations. We conclude the appendix with the example of the
computation of $IH^{\bullet}(\Sat, \QQ)$ for $g=6,7$, and relatively simple
formulas to compute the polynomials
\[ \sum_k T^k \dim IH^k(\Sat, \VV_{\lambda}) \]
for all values of $(g, \lambda)$ such that the corresponding set of substitutes
for Arthur-Langlands parameters of conductor on is known (currently $g \leq 7$
and arbitrary $\lambda$ and all pairs $(g, \lambda)$ with $g + \lambda_1 \leq
11$).

Let us recall that for $n = 3,4,5 \mod 8$ there is a (unique by
\cite[Proposition 2.1]{Gross_gpsZ}) reductive group $G$ over $\ZZ$ such that
$G_{\RR} \simeq\SO(n-2,2)$. Such $G$ is a special orthogonal group of a lattice,
for example $E_8 \oplus H^{\oplus 2}$ where $H$ is a hyperbolic lattice. If $K$
is a maximal compact subgroup of $G(\RR)$, we can also consider the Hermitian
locally symmetric space
\[ G(\ZZ) \backslash G(\RR) / K^0 = G(\ZZ)^0 \backslash G(\RR)^0 / K^0 \]
where $G(\RR)^0$ (resp.\ $K^0$) is the identity component of $G(\RR)$ (resp.\
$K$) and $G(\ZZ)^0 = G(\ZZ) \cap G(\RR)^0$. Then everything explained in this
appendix also applies to this situation, except for the simplification in
Proposition \ref{prop:simplification_Tell_sc} which can only be applied to the
simply connected cover of $G$. Using \cite[\S 4.3]{ChenevierRenard} one can see
that this amounts to considering $(\mathfrak{so}_n, SO(2)
\times SO(n-2))$-cohomology instead of $(\mathfrak{so}_n, S(O(2) \times
O(n-2)))$-cohomology as in \cite{Taibi_dimtrace}, and this simply multiplies
Euler characteristics by $2$. In fact the analogue of Section
\ref{sec:app_Shimura} is much simpler for special orthogonal groups, since they
do occur in Shimura data (of abelian type).

To be complete we mention that Arthur's endoscopic classification in
\cite{Arthur_book} is conditional on several announced results which, to the
best of our knowledge, are not yet available (see \cite[\S
1.3]{Taibi_dimtrace}).

\medskip
    {\small {\bf Acknowledgements.}
Klaus Hulek presented his joint work with Sam Grushevsky \cite{HGIH} at the
Oberwolfach workshop ``Moduli spaces and Modular forms'' in April 2016. During
this workshop Dan Petersen pointed out that $IH^{\bullet}(\Sat)$ can also be
computed using \cite{Taibi_dimtrace}. I thank Dan Petersen, the organizers of
this workshop (Jan Hendrik Bruinier, Gerard van der Geer and Valery Gritsenko)
and the Mathematisches Institut Oberwolfach. I also thank Eduard Looijenga for
kindly answering questions related to Zucker's conjecture.
    }
    
\subsection{Evaluation of a trace formula}
\label{sec:app_TF_comp}

Our first goal is to prove the following result.

\begin{prop} \label{prop:eA4}
	We have $e(\ab[4]) = 9$.
\end{prop}

This number is a byproduct of the explicit computation in \cite{Taibi_dimtrace}
of
\begin{equation} \label{e:Euler_L2}
	e_{(2)}(\ab, \VV_{\lambda}) := \sum_i (-1)^i \dim_{\RR} H^i_{(2)}(\ab,
	\VV_{\lambda})
\end{equation}
for arbitrary irreducible algebraic representations $V_{\lambda}$ of $\Sp_{2g}$.
Each representation $V_{\lambda}$ is defined over $\QQ$, and $L^2$-cohomology is
defined with respect to an admissible inner product on $\RR \otimes_{\QQ}
V_{\lambda}$. In particular $H^{\bullet}_{(2)}(\ab, \VV_{\lambda})$ is a graded
real vector space (for arbitrary arithmetic symmetric spaces the representation
$V_{\lambda}$ may not be defined over $\QQ$ and so in general $L^2$-cohomology
is only naturally defined over $\CC$). Recall that by Zucker's conjecture
\eqref{e:Euler_L2} is also equal to
\[	\sum_i (-1)^i \dim IH^i(\Sat, \VV_{\lambda}). \]
To evaluate the Euler characteristic \eqref{e:Euler_L2} we use Arthur's
$L^2$-Lefschetz trace formula \cite{Arthur_L2}. This is a special case since
this is the alternating trace of the unit in the unramified Hecke algebra on
these cohomology groups. Arthur obtained this formula by specializing his more
general invariant trace formula. In general Arthur's invariant trace formula
yields transcendental values, but for particular functions at the real place
(pseudo-coefficients of discrete series representations) Arthur obtained a
simplified expression that only takes rational values. Goresky, Kottwitz and
MacPherson \cite{GoKoMPh}, \cite{GoreskyMacPherson} gave a different proof of
this formula, by a topological method. In fact they obtained more generally a
trace formula for \emph{weighted cohomology}
\cite{GoreskyHarderMacPherson_weighted}, the case of a lower middle or upper
middle weight profile on a Hermitian locally symmetric space recovering
intersection cohomology of the Baily-Borel compactification. We shall also use
split reductive groups over $\QQ$ other than $\Sp_{2g}$ below, which will be of
equal rank at the real place but do not give rise to Hermitian symmetric spaces.
For this reason it is reassuring that Nair \cite{Nair_weighted} proved that in
general weighted cohomology groups coincide with Franke's weighted $L^2$
cohomology groups \cite{Franke} defined in terms of automorphic forms. The case
of usual $L^2$ cohomology corresponds to the lower and upper middle weight
profiles in \cite{GoreskyHarderMacPherson_weighted}. In particular Nair's result
implies that \cite{GoKoMPh} is a generalization of \cite{Arthur_L2}.

In our situation the trace formula can be written
\[ e_{(2)}(\ab, \VV_{\lambda}) = \sum_M T(\Sp_{2g}, M, \lambda) \]
where the sum is over conjugacy classes of Levi subgroups $M$ of $\Sp_{2g}$
which are $\RR$-cuspidal, i.e.\ isomorphic to $\GL_1^a \times \GL_2^c \times
\Sp_{2d}$ with $a+2c+d = g$. The right hand side is traditionally called the
geometric side, although this terminology is confusing in the present context.
The most interesting term in the sum is the elliptic part
$T_{\op{ell}}(\Sp_{2g}, \lambda) := T(\Sp_{2g}, \Sp_{2g}, \lambda)$ which is
defined as
\begin{equation} \label{equ:Tell}
	T_{\op{ell}}(\Sp_{2g}, \lambda) = \sum_{c \in C(\Sp_{2g})} m_c \tr \left( c
\,\middle|\, V_{\lambda} \right).
\end{equation}
Here $C(\Sp_{2g})$ is the finite set of torsion $\RR$-elliptic elements in
$\Sp_{2g}(\QQ)$ up to conjugation in $\Sp_{2g}(\overline{\QQ})$. These can be
simply described by certain products of degree $2n$ of cyclotomic polynomials.
The rational numbers $m_c$ are ``masses'' (in the sense of the mass formula, so
it would be more correct to call them ``weights'') computed adelically,
essentially as products of local orbital integrals (at all prime numbers) and
global terms (involving Tamagawa numbers and values of certain Artin L-functions
at negative integers). We refer the reader to \cite{Taibi_dimtrace} for details.
Let us simply mention that for $c = \pm 1 \in C(\Sp_{2g})$, the local orbital
integrals are all equal to $1$, and $m_c$ is the familiar product $\zeta(-1)
\zeta(-3) \dots \zeta(1-2g)$. The appearance of other terms in
$T_{\op{ell}}(\Sp_{2g}, \lambda)$ corresponding to non-central elements in
$C(\Sp_{2g})$ is explained by the fact that the action of $\Sp_{2g}(\ZZ) / \{
\pm 1 \}$ on $\HH_g$ is not free.

To evaluate $T_{\op{ell}}(\Sp_{2g}, \lambda)$ explicitly, the main difficulty
consists in computing the local orbital integrals. An algorithm was given in
\cite[\S 3.2]{Taibi_dimtrace}. In practice these are computable (by a computer)
at least for $g \leq 7$. For $g = 2$ they were essentially computed by Tsushima
in \cite{Tsushima_deg2}. For $g = 3$ they could also be computed by a
(dedicated) human being. See \cite[Table 9]{Taibi_dimtrace} for $g=3$, and
\cite{onlinelistdimtrace} for higher $g$. The following table contains the
number of masses in each rank, taking into account that $m_{-c} = m_c$.

\[ \begin{array}{c|rrrrrrr}
	g              & 1 & 2  & 3  & 4  & 5   & 6   & 7    \\
	\hline
	\op{card} \left( C(\Sp_{2g}) / \{ \pm 1 \} \right) & 3 & 12 & 32 & 92 & 219 &
	530 & 1158
	\end{array} \]

In general the elliptic part of the geometric side of the $L^2$-Lefschetz trace
formula does not seem to have any spectral or cohomological meaning, but for
unit Hecke operators and simply connected groups, such as $\Sp_{2g}$, it turns
out that it does.

\begin{prop} \label{prop:simplification_Tell_sc}
	Let $G$ be a simply connected reductive group over $\QQ$. Assume that
	$G_{\RR}$ has equal rank, i.e.\ $G_{\RR}$ admits a maximal torus (defined over
	$\RR$) which is anisotropic, and that $G_{\RR}$ is not anisotropic, i.e.\
	$G(\RR)$ is not compact. Let $K_f$ be a compact open subgroup of $G(\AQ_f)$
	and let $\Gamma = G(\QQ) \cap K_f$. Let $K_{\infty}$ be a maximal compact
	subgroup of $G(\RR)$ (which is connected). Then for any irreducible algebraic
	representation $V_{\lambda}$ of $G(\CC)$, in Arthur's $L^2$-Lefschetz trace
	formula
	\[ e_{(2)}( \Gamma \backslash G(\RR) / K_{\infty}, \VV_{\lambda}) = \sum_M
	T(G, K_f, M, \lambda) \]
	where the sum is over $G(\QQ)$-conjugacy classes of cuspidal Levi subgroups of
	$G$, the elliptic term
	\[ T_{\op{ell}}(G, K_f, \lambda) := T(G, K_f, G, \lambda) \]
	is equal to
	\[ e_c( \Gamma \backslash G(\RR) / K_{\infty}, \VV_{\lambda}) = \sum_i (-1)^i
	\dim H^i_c( \Gamma \backslash G(\RR) / K_{\infty}, \VV_{\lambda}). \]
\end{prop}

\begin{proof}
	Note that by strong approximation (using that $G(\RR)$ is not compact), the
	natural inclusion $\Gamma \backslash G(\RR) / K_{\infty} \rightarrow G(\QQ)
	\backslash G(\AQ_f) / K_{\infty} K_f$ is an isomorphism between orbifolds.

	We claim that in the formula \cite[\S 7.17]{GoKoMPh} for $e_c( \Gamma
	\backslash G(\RR) / K_{\infty}, \VV_{\lambda})$, every term corresponding to
	$M \neq G$ vanishes. Note that in \cite{GoKoMPh} the \emph{right} action of
	Hecke operators is considered: see \cite[\S 7.19]{GoKoMPh}. Thus to recover
	the trace of the usual \emph{left} action of Hecke operators one has to
	exchange $E$ (our $V_{\lambda}$) and $E^*$ in \cite[\S 7.17]{GoKoMPh}. Since
	we are only considering a unit Hecke operator, the orbital integrals at finite
	places denoted $O_{\gamma}(f_M^{\infty})$ in \cite[Theorem 7.14.B]{GoKoMPh}
	vanish unless $\gamma \in M(\QQ)$ is power-bounded in $M(\QQ_p)$ for every
	prime $p$. Recall that $\gamma$ is also required to be elliptic in $M(\RR)$,
	so that by the adelic product formula this condition on $\gamma$ at all finite
	places implies that $\gamma$ is also power-bounded in $M(\RR)$. Thus it is
	enough to show that for any cuspidal Levi subgroup $M$ of $G_{\RR}$ distinct
	from $G_{\RR}$, and any power-bounded $\gamma \in M(\RR)$, we have
	$\Phi_M(\gamma, V_{\lambda}) = 0$ (where $\Phi_M$ is defined on p.498 of
	\cite{GoKoMPh}). Choose an elliptic maximal torus $T$ in $M_{\RR}$ such that
	$\gamma \in T(\RR)$. Since the character of $V_{\lambda}$ is already a
	continuous function on $M(\RR)$, it is enough to show that there is a root
	$\alpha$ of $T$ in $G$ not in $M$ such that $\alpha(\gamma) = 1$. All roots of
	$T$ in $M$ are imaginary, so it is enough to show that there exists a real
	root $\alpha$ of $T$ in $G$ such that $\alpha(\gamma) = 1$. Since $\gamma$ is
	power-bounded in $M(\RR)$, for any real root of $T$ in $G$ we have
	$\alpha(\gamma) \in \{ \pm 1 \}$. Thus it is enough to show that there is a
	real root $\alpha$ such that $\alpha(\gamma) > 0$. This follows (for $M \neq
	G_{\RR}$) from the argument at the bottom of p.499 in \cite{GoKoMPh} (this is
	were the assumption that $G$ is simply connected is used).
\end{proof}

The proposition is a generalization of \cite{Harder_GaussBonnet} to the orbifold
case ($\Gamma$ not neat) with non-trivial coefficients, but note that Harder's
formula is used in \cite{GoKoMPh}.

In particular for any $g \geq 1$ and dominant weight $\lambda$ we simply have
$T_{\op{ell}}(\Sp_{2g}, \lambda) = e_c(\ab, \VV_{\lambda}) = e(\ab,
\VV_{\lambda})$ by Poincar\'e duality and self-duality of $V_{\lambda}$.
As a special case we have the simple formula $e(\ab) = \sum_{c \in C(\Sp_{2g})}
m_c$ and Proposition \ref{prop:eA4} follows (with the help of a computer). See
\cite{onlinelistdimtrace} for tables of masses $m_c$, where the source code
computing these masses (using \cite{sage}) can also be found.

In the following table we record the value of $e(\ab)$ for small $g$.

\[ \begin{array}{l|rrrrrrrrr}
	g       & 1 & 2  & 3  & 4  & 5   & 6   & 7   & 8   & 9 \\
	\hline
	e(\ab)  & 1 & 2  & 5  & 9  & 18  & 46  & 104 & 200 & 528
	\end{array} \]

The Euler characteristic of $L^2$-cohomology can also be evaluated explicitly.
Theorem 3.3.4 in \cite{Taibi_dimtrace} expresses $e_{(2)}(\ab, \VV_{\lambda})$
in a (relatively) simple manner from $T_{\op{ell}}(\Sp_{2g'}, \lambda')$ for $g'
\leq g$ and dominant weights $\lambda'$. Hence $e_{(2)}(\ab, \VV_{\lambda})$ can
be derived from tables of masses, for any $\lambda$. Of course this does not
directly yield dimensions of individual cohomology groups. Fortunately, Arthur's
endoscopic classification of automorphic representations for $\Sp_{2g}$ allows
us to write this Euler characteristic as a sum of two contributions: ``old''
contributions coming from automorphic representations for groups of lower
dimension, and new contributions which only contribute to middle degree. Thus it
is natural to try to compute old contributions by induction. As we shall see
below, they can be described combinatorially from certain self-dual level one
automorphic cuspidal representations of general linear groups over $\QQ$.

\subsection{Arthur's endoscopic classification in level one}
\label{sec:app_mult_formula}

We will explain the decomposition of $H^{\bullet}_{(2)}(\ab, \VV_{\lambda})$
that can be deduced from Arthur's endoscopic classification. These real graded
vector spaces are naturally endowed with a real Hodge structure, a Lefschetz
operator and a compatible action of a commutative Hecke algebra. This last
action will make the decomposition canonical. For this reason we will recall
what is known about this Hecke action. We denote $\AQ$ for the adele ring of
$\QQ$.

\begin{df}
	Let $G$ be a reductive group over $\ZZ$ (in the sense of \cite[Expos\'e XIX,
	D\'efinition 2.7]{SGA3-III}).
\begin{enumerate}
	\item If $p$ is a prime, let $\calH_p^{\op{unr}}(G)$ be the commutative
		convolution algebra (``Hecke algebra'') of functions $G(\ZZ_p) \backslash
		G(\QQ_p) / G(\ZZ_p) \rightarrow \QQ$ having finite support.
	\item Let $\calH^{\op{unr}}_f(G)$ be the commutative algebra of functions
		\[ G(\widehat{\ZZ}) \backslash G(\AQ) / G(\widehat{\ZZ}) \rightarrow \QQ \]
		having finite support, so that $\calH^{\op{unr}}_f(G)$ is the restricted
		tensor product $\bigotimes_p' \calH_p^{\op{unr}}(G)$.
\end{enumerate}
\end{df}

Recall the Langlands dual group $\widehat{G}$ of a reductive group $G$, which we
consider as a split reductive group over $\QQ$. We will mainly consider the
following cases.

\[ \begin{array}{c|cccc}
	G           & \GL_N & \Sp_{2g}   & \SO_{4g} & \SO_{2g+1} \\
	\hline
	\widehat{G} & \GL_N & \SO_{2g+1} & \SO_{4g} & \Sp_{2g}
	\end{array} \]

Recall from \cite{Satake_iso}, \cite{Gross_Satake} the Satake isomorphism: for
$F$ an algebraically closed field of characteristic zero and $G$ a reductive
group over $\ZZ$, if we choose a square root of $p$ in $F$ then $\QQ$-algebra
morphisms $\calH^{\op{unr}}_p(G) \rightarrow F$ correspond naturally and
bijectively to semisimple conjugacy classes in $\widehat{G}(F)$.

If $\pi$ is an automorphic cuspidal representation of $\GL_N(\AQ)$, then it
admits a decomposition as a restricted tensor product $\pi = \pi_{\infty}
\otimes \bigotimes'_p \pi_p$, where the last restricted tensor product is over
all prime numbers $p$. Assume moreover that $\pi$ has \emph{level one}, i.e.\
that $\pi_p^{\GL_N(\ZZ_p)} \neq 0$ for any prime $p$. Then $\pi$ has the
following invariants:
\begin{enumerate}
	\item the infinitesimal character $\op{ic}(\pi_{\infty})$, which is a
		semisimple conjugacy class in $\mathfrak{gl}_N(\CC) = M_N(\CC)$ obtained
		using the Harish-Chandra isomorphism \cite{HarishChandra_iso},
	\item for each prime number $p$, the Satake parameter $c(\pi_p)$ of the
		unramified representation $\pi_p$ of $\GL_N(\QQ_p)$, which is a semisimple
		conjugacy class in $\GL_N(\CC)$ corresponding to the character by which
		$\calH_p^{\op{unr}}(\GL_N)$ acts on the complex line $\pi_p^{\GL_N(\ZZ_p)}$.
\end{enumerate}

We now introduce three families of automorphic cuspidal representations for
general linear groups that will be exactly those contributing to intersection
cohomology of $\ab$'s.

\begin{df} \label{def:three_sets}
\begin{enumerate}
	\item For $g \geq 0$ and integers $w_1 > \dots > w_g > 0$, let $O_o(w_1,
		\dots, w_g)$ be the set of self-dual level one automorphic cuspidal
		representations $\pi = \pi_{\infty} \otimes \pi_f$ of $\GL_{2g+1}(\AQ)$ such
		that $\op{ic}(\pi_{\infty})$ has eigenvalues $w_1 > \dots > w_g > 0 > -w_g >
		\dots > -w_1$. For $\pi \in O_o(w_1, \dots, w_g)$ we let $\widehat{G_{\pi}}
		= \SO_{2g+1}(\CC)$.
	\item For $g \geq 1$ and integers $w_1 > \dots > w_{2g} > 0$, let $O_e(w_1,
		\dots, w_{2g})$ be the set of self-dual level one automorphic cuspidal
		representations $\pi = \pi_{\infty} \otimes \pi_f$ of $\GL_{4g}(\AQ)$ such
		that $\op{ic}(\pi_{\infty})$ has eigenvalues $w_1 > \dots > w_{2g} > -w_{2g}
		> \dots > -w_1$. For $\pi \in O_e(w_1, \dots, w_{2g})$ we let
		$\widehat{G_{\pi}} = \SO_{4g}(\CC)$.
	\item For $n \geq 1$ and $w_1 > \dots > w_g > 0$ with $w_i \in 1/2 + \ZZ$, let
		$S(w_1, \dots, w_g)$ be the set of self-dual level one automorphic cuspidal
		representations $\pi = \pi_{\infty} \otimes \pi_f$ of $\GL_{2g}(\AQ)$ such
		that $\op{ic}(\pi_{\infty})$ has eigenvalues $w_1 > \dots > w_g > -w_g >
		\dots > -w_1$. For $\pi \in S(w_1, \dots, w_g)$ we let $\widehat{G_{\pi}} =
		\Sp_{2g}(\CC)$.
\end{enumerate}
\end{df}
These sets are all finite by \cite[Theorem 1]{HCAuto}, and $O_o$ (resp.\ $O_e$,
$S$) is short for ``odd orthogonal'' (resp.\ ``even orthogonal'',
``symplectic''). A fact related to vanishing of cohomology with coefficients in
$\VV_{\lambda}$ for $w(\lambda)$ odd is that $O_o(w_1, \dots, w_g) = \emptyset$
if $w_1 + \dots + w_g \neq g(g+1)/2 \mod 2$ and $O_e(w_1, \dots, w_{2g}) =
\emptyset$ if $w_1 + \dots + w_{2g} \neq g \mod 2$. See \cite[Remark
4.1.6]{Taibi_dimtrace} or \cite[Proposition 1.8]{ChenevierRenard}.

\begin{rem} \label{rem:three_sets_smallg}
	For small $g$ the sets in Definition \ref{def:three_sets} are completely
	described in terms of level one (elliptic) eigenforms, due to accidental
	isomorphisms between classical groups in small rank (see \cite[\S
	6]{Taibi_dimtrace} for details).
	\begin{enumerate}
		\item For $k>0$ the set $S(k-\frac{1}{2})$ is naturally in bijection with
			the set of normalized eigenforms of weight $2k$ for $\SL_2(\ZZ)$, and
			$O_o(2k-1) \simeq S(k-\frac{1}{2})$.
		\item For integers $w_1 > w_2 > 0$ such that $w_1+w_2$ is odd, $O_e(w_1,
			w_2) \simeq S(\frac{w_1+w_2}{2}) \times S(\frac{w_1-w_2}{2})$.
	\end{enumerate}
\end{rem}

Let us recall some properties of the representations appearing in Definition
\ref{def:three_sets}.

Let $\QQ^{\op{real}}$ be the maximal totally real algebraic extension of $\QQ$
in $\CC$. Then $\QQ^{\op{real}}$ is an infinite Galois extension of $\QQ$ which
contains $\sqrt{p} > 0$ for any prime $p$.
\begin{thm}
	For $\pi$ as in Definition \ref{def:three_sets} there exists a finite
	subextension $E$ of $\QQ^{\op{real}} / \QQ$ such that for any prime number
	$p$, the characteristic polynomial of $p^{w_1} c(\pi_p)$ has coefficients in
	$E$. Moreover $c(\pi_p)$ is compact (i.e.\ power-bounded).
\end{thm}
\begin{proof}
	That there exists a finite extension $E$ of $\QQ$ satisfying this condition is
	a special case of \cite[Th\'eor\`eme 3.13]{Clozel_motifs}. The fact that it
	can be taken totally real follows from unitarity and self-duality of $\pi$.

	The last statement is a consequence of \cite{Shin_Galois} and
	\cite{Clozel_purity}.
\end{proof}
Let $E(\pi)$ be the smallest such extension. Then $\pi_f$ is defined over
$E(\pi)$, and this structure is unique up to $\CC^{\times} / E(\pi)^{\times}$.
There is an action of the Galois group $\op{Gal}(\QQ^{\op{real}} / \QQ)$ on
$O_o(w_1, \dots)$ (resp.\ $O_e(w_1, \dots)$, $S(w_1, \dots)$): if $\pi =
\pi_{\infty} \otimes \pi_f$, $\sigma(\pi) := \pi_{\infty} \otimes \sigma(\pi_f)$
belongs to the same set. Dually we have $c(\sigma(\pi)_p) = p^{-w_1} \sigma(
p^{w_1} c(\pi_p))$. For $\pi \in O_o(w_1, \dots)$ or $O_e(w_1, \dots)$ the power
of $p$ is not necessary since $w_1 \in \ZZ$.

In all three cases $c(\pi_p)$ can be lifted to a semisimple conjugacy class in
$\widehat{G_{\pi}}(\CC)$, uniquely except in the second case, where it is unique
only up to conjugation in $\Or_{4g}(\CC)$. This is elementary. We abusively
denote this conjugacy class by $c(\pi_p)$.

We now consider the Archimedean place of $\QQ$, which will be of particular
importance for real Hodge structures. Recall that the Weil group of $\RR$ is
defined as the non-trivial extension of $\op{Gal}(\CC / \RR)$ by $\CC^{\times}$.
If $H$ is a complex reductive group and $\varphi : \CC^{\times} \rightarrow
H(\CC)$ is a continuous semisimple morphism, there is a maximal torus $T$ of $H$
such that $\varphi$ factors through $T(\CC)$ and takes the form $z \mapsto
z^{\tau_1} \bar{z}^{\tau_2}$ for uniquely determined $\tau_1, \tau_2 \in X_*(T)
\otimes_{\ZZ} \CC$ such that $\tau_1 - \tau_2 \in X_*(T)$. Here $X_*(T)$ is the
group of cocharacters of $T$ and $z^{\tau_1} \bar{z}^{\tau_2}$ is defined as
$(z/|z|)^{\tau_1-\tau_2} |z|^{\tau_1+\tau_2}$. We call the $H(\CC)$-conjugacy
class of $\tau_1$ in $\mathfrak{h} = \Lie(H)$ (complex analytic Lie algebra) the
\emph{infinitesimal character} of $\varphi$, denoted $\op{ic}(\varphi)$. If
$\varphi : W_{\RR} \rightarrow H(\CC)$ is continuous semisimple we let
$\op{ic}(\varphi) = \op{ic}(\varphi|_{\CC^{\times}})$.
		
In all three cases in Definition \ref{def:three_sets} there is a continuous
semisimple morphism $\varphi_{\pi_{\infty}} : W_{\RR} \rightarrow
\widehat{G_{\pi}}(\CC)$ such that for any $z \in \CC^{\times}$,
$\varphi_{\pi_{\infty}}(z)$ has eigenvalues \[ \begin{cases} (z/\bar{z})^{\pm
		w_1}, \dots, (z/\bar{z})^{\pm w_g}, 1 & \text{ if } \pi \in O_o(w_1, \dots,
		w_g), \\ (z/\bar{z})^{\pm w_1}, \dots, (z/\bar{z})^{\pm w_{2g}} & \text{ if
		} \pi \in O_e(w_1, \dots, w_{2g}), \\ (z/|z|)^{\pm 2 w_1}, \dots,
(z/|z|)^{\pm 2 w_g} & \text{ if } \pi \in S(w_1, \dots, w_g), \end{cases} \] in
the standard representation of $\widehat{G_{\pi}}$. The parameter
$\varphi_{\pi_{\infty}}$ is characterized up to conjugacy by this property
except in the second case where it is only characterized up to
$\Or_{4g}(\CC)$-conjugacy. To rigidify the situation we choose a semisimple
conjugacy class $\tau_{\pi}$ in the Lie algebra of $\widehat{G_{\pi}}$ whose
image via the standard representation has eigenvalues $\pm w_1, \dots, \pm w_g,
0$ (resp.\ $\pm w_1, \dots, \pm w_{2g}$, resp.\ $\pm w_1, \dots, \pm w_g$). Then
the pair $(\widehat{G_{\pi}}, \tau_{\pi})$ is well-defined up to isomorphism
unique up to conjugation by $\widehat{G_{\pi}}$, since in the even orthogonal
case $\tau_{\pi}$ is not fixed by an outer automorphism of $\widehat{G_{\pi}}$.
Up to conjugation by $\widehat{G_{\pi}}$ there is a unique
$\varphi_{\pi_{\infty}} : W_{\RR} \rightarrow \widehat{G_{\pi}}(\CC)$ as above
and such that $\op{ic}(\varphi_{\pi_{\infty}}) = \tau_{\pi}$. In all cases the
centralizer of $\varphi_{\pi_{\infty}}(W_{\RR})$ in $\widehat{G_{\pi}}(\CC)$ is
finite.

Let us now indicate how the general definition of substitutes for
Arthur-Langlands parameters in \cite{Arthur_book} specializes to the case at
hand.

\begin{df} \label{def:AL_param}
Let $V_{\lambda}$ be an irreducible algebraic representation of $\Sp_{2g}$,
given by the dominant weight $\lambda = (\lambda_1 \geq \dots \geq \lambda_g
\geq 0)$. Let $\rho$ be half the sum of the positive roots for $\Sp_{2g}$, and
$\tau = \lambda + \rho = (w_1 > \dots > w_g > 0)$ where $w_i = \lambda_i + n+1-i
\in \ZZ$, which we can see as the regular semisimple conjugacy class in
$\mathfrak{so}_{2g+1}(\CC)$. Let $\Psi_{\op{disc}}^{\op{unr}, \tau}(\Sp_{2g})$
be the set of pairs $(\psi_0, \{ \psi_1, \dots, \psi_r \} )$ with $r \geq 0$ and
such that
\begin{enumerate}
	\item $\psi_0 = (\pi_0, d_0)$ where $\pi_0 \in O_o(w_1^{(0)}, \dots,
		w_{g_0}^{(0)})$ and $d_0 \geq 1$ is an odd integer,
	\item The $\psi_i$'s, for $i \in \{1, \dots, r\}$, are distinct pairs $(\pi_i,
		d_i)$ where $d_i \geq 1$ is an integer and $\pi_i \in O_e(w_1^{(i)}, \dots,
		w_{g_i}^{(i)})$ with $g_i$ even (resp.\ $\pi_i \in S(w_1^{(i)}, \dots,
		w_{g_i}^{(i)})$) if $d_i$ is odd (resp.\ even).
	\item $2g+1 = (2g_0+1) d_0 + \sum_{i=1}^r 2 g_i d_i$,
	\item The sets
		\begin{enumerate}
			\item $\{ \frac{d_0-1}{2}, \frac{d_0-3}{2}, \dots, 1 \}$,
			\item $\{ w_1^{(0)} + \frac{d_0-1}{2} - j, \dots, w^{(0)}_{g_0} +
				\frac{d_0-1}{2} -j \}$ for $j \in \{0, \dots, d_0 - 1\}$,
			\item $\{ w_1^{(i)} + \frac{d_i-1}{2} - j, \dots, w_{g_i}^{(i)} +
				\frac{d_i-1}{2} -j \}$ for $i \in \{ 1, \dots, r \}$ and $j \in \{ 0,
				\dots, d_i-1 \}$
		\end{enumerate}
		are disjoint and their union equals $\{ w_1, \dots, w_g \}$.
\end{enumerate}
We will write more simply
\[ \psi = \psi_0 \boxplus \dots \boxplus \psi_r = \pi_0[d_0] \boxplus \dots
	\boxplus \pi_r[d_r]. \]
This could be defined as an ``isobaric sum'', but for the purpose of this
appendix we can simply consider this expression as a formal unordered sum. If
$\pi_0$ is the trivial representation of $\GL_1(\AQ)$, we write $[d_0]$ for
$1[d_0]$, and when $d=1$ we simply write $\pi$ for $\pi[1]$.
\end{df}

\begin{exa}
	\begin{enumerate}
		\item For any $g \geq 1$, $[2g+1] \in \Psi_{\op{disc}}^{\op{unr},
			\rho}(\Sp_{2g})$.
		\item $[9] \boxplus \Delta_{11}[2] \in \Psi_{\op{disc}}^{\op{unr},
			\rho}(\Sp_{12})$ where $\Delta_{11} \in S(\frac{11}{2})$ is the
			automorphic representation of $\GL_2(\AQ)$ corresponding to the Ramanujan
			$\Delta$ function.
		\item For $g \geq 1$, $\tau = (w_1 > \dots > w_g > 0)$, for any $\pi \in
			O_o(w_1, \dots, w_g)$ we have $\pi = \pi[1] \in
			\Psi_{\op{disc}}^{\op{unr}, \tau}(\Sp_{2g})$.
	\end{enumerate}
\end{exa}

\begin{df} \label{def:Lpsi}
	\begin{enumerate}
		\item For $\psi = \pi_0[d_0] \boxplus \dots \boxplus \pi_r[d_r] \in
			\Psi_{\op{disc}}^{\op{unr}, \tau}(\Sp_{2g})$, let $\calL_{\psi} =
			\prod_{i=0}^r \widehat{G_{\pi_i}}(\CC)$. Let $\dot{\psi} : \calL_{\psi}
			\times \SL_2(\CC) \rightarrow \SO_{2n+1}(\CC)$ be a morphism such that
			composing with the standard representation gives $\bigoplus_{0 \leq i \leq
			r} \op{Std}_{\widehat{G_{\pi_i}}} \otimes \nu_{d_i}$ where
			$\op{Std}_{\widehat{G_{\pi_i}}}$ is the standard representation of
			$\widehat{G_{\pi_i}}$ and $\nu_{d_i}$ is the irreducible representation of
			$\SL_2(\CC)$ of dimension $d_i$. Then $\dot{\psi}$ is well-defined up to
			conjugation by $\SO_{2n+1}(\CC)$.
		\item Let $\calS_{\psi}$ be the centralizer of $\dot{\psi}$ in
			$\SO_{2g+1}(\CC)$, which is isomorphic to $(\ZZ/2\ZZ)^r$. A basis is given
			by $(s_i)_{1 \leq i \leq r}$ where $s_i$ is the image by $\dot{\psi}$ of
			the non-trivial element in the center of $\widehat{G_{\pi_i}}$.
		\item Let $\psi_{\infty}$ be the morphism $W_{\RR} \times \SL_2(\CC)
			\rightarrow \SO_{2g+1}(\CC)$ obtained by composing $\dot{\psi}$ with the
			morphisms $\varphi_{\pi_{i, \infty}} : W_{\RR} \rightarrow
			\widehat{G_{\pi_i}}(\CC)$. The centralizer $\calS_{\psi_{\infty}}$
			of $\psi_{\infty}$ in $\SO_{2g+1}(\CC)$ contains $\calS_{\psi}$ and
			is isomorphic to $(\ZZ / 2 \ZZ )^x$ where $x = \sum_{i=1}^r g_i$.
		\item For $p$ a prime let $c_p(\psi)$ be the image of $((c(\pi_{i,p}))_{0
			\leq i \leq r}, \op{diag}(p^{1/2}, p^{-1/2}))$ by $\dot{\psi}$, a
			well-defined semisimple conjugacy class in $\SO_{2g+1}(\CC)$. Let
			$\chi_p(\psi) : \calH_p^{\op{unr}}(\Sp_{2g}) \rightarrow \CC$ be the
			associated character.
		\item Let $\chi_f(\psi) = \prod_p \chi_p(\psi) :
			\calH_f^{\op{unr}}(\Sp_{2g}) \rightarrow \CC$ be the product of the
			$\chi_p(\psi)$'s. It takes values in the smallest subextension $E(\psi)$
			of $\QQ^{\op{real}}$ containing $E(\pi_0), \dots, E(\pi_r)$, which is also
			finite over $\QQ$. For any prime $p$ the characteristic polynomial of
			$c_p(\psi)$ has coefficients in $E(\psi)$. In particular we have a
			continuous action of $\op{Gal}(\QQ^{\op{real}} / \QQ)$ on
			$\Psi_{\op{disc}}^{\op{unr}, \tau}(\Sp_{2g})$ which is compatible with
			$\chi_f$.
	\end{enumerate}
\end{df}

The following theorem is a consequence of \cite{JacquetShalika_Euler2}.
\begin{thm} \label{thm:mult_one}
	For any $g \geq 1$ the map $(\lambda, \psi \in \Psi_{\op{disc}}^{\op{unr},
	\tau}(\Sp_{2g}) ) \mapsto \chi_f(\psi)$ is injective. 
\end{thm}

The last condition in Definition \ref{def:AL_param} is explained by compatibility
with infinitesimal characters, stated after the following definition.
\begin{df} \label{def:delta}
	Let
	\begin{align*}
		\delta_{\infty} : \CC^{\times} & \longrightarrow \CC^{\times} \times
			\SL_2(\CC) \\
		z & \longmapsto (z, \op{diag}(||z||^{1/2}, ||z||^{-1/2})).
	\end{align*}
	For a complex reductive group $H$ and a morphism $\psi_{\infty} : \CC^{\times}
	\times \SL_2(\CC) \rightarrow H(\CC)$ which is continuous semisimple and
	algebraic on $\SL_2(\CC)$, let $\op{ic}(\psi_{\infty}) = \op{ic}(\psi_{\infty}
	\circ \delta_{\infty})$. Similarly, if $\psi_{\infty} : W_{\RR} \times
	\SL_2(\CC) \rightarrow H(\CC)$, let $\op{ic}(\psi_{\infty}) =
	\op{ic}(\psi_{\infty}|_{\CC^{\times}})$.
\end{df}
For $\psi \in \Psi_{\op{disc}}^{\op{unr}, \tau}(\Sp_{2g})$, we have
$\op{ic}(\psi_{\infty}) = \tau$ (equality between semisimple conjugacy classes
in $\mathfrak{so}_{2g+1}(\CC)$), and this explains the last condition in
Definition \ref{def:AL_param}.

For $\psi$ as above Arthur constructed \cite[Theorem 1.5.1]{Arthur_book} a
finite set $\Pi(\psi_{\infty})$ of irreducible unitary representations of
$\Sp_{2g}(\RR)$ and a map $\Pi(\psi_{\infty}) \rightarrow
\calS_{\psi_{\infty}}^{\vee}$, where $A^{\vee} = \Hom(A, \CC^{\times})$. We
simply denote this map by $\pi_{\infty} \mapsto \langle \cdot, \pi_{\infty}
\rangle$. Arthur also defined a character $\epsilon_{\psi}$ of $\calS_{\psi}$,
in terms of symplectic root numbers. We do not recall the definition, but note
that for everywhere unramified parameters considered here, this character can be
computed easily from the infinitesimal characters of the $\pi_i$'s (see \cite[\S
3.9]{ChenevierRenard}).

We can now formulate the specialization of \cite[Theorem 1.5.2]{Arthur_book} to
level one and algebraic regular infinitesimal character, and its consequence for
$L^2$-cohomology thanks to \cite{BorelCasselman_L2}.
\begin{thm} \label{thm:mult_formula}
	Let $g \geq 1$. Let $V_{\lambda}$ be an irreducible algebraic representation
	of $\Sp_{2g}$ with dominant weight $\lambda$. Let $\tau = \lambda + \rho$. The
	part of the discrete automorphic spectrum for $\Sp_{2g}$ having level
	$\Sp_{2g}(\widehat{\ZZ})$ and infinitesimal character $\tau$ decomposes as a
	completed orthogonal direct sum
	\[ L^2_{\op{disc}}( \Sp_{2g}(\QQ) \backslash \Sp_{2g}(\AQ) /
		\Sp_{2g}(\widehat{\ZZ}))^{\op{ic} = \tau} \simeq \bigoplus_{\psi \in
		\Psi_{\op{disc}}^{\op{unr}, \tau}(\Sp_{2g})}
		\bigoplus_{\substack{\pi_{\infty} \in \Pi(\psi_{\infty}) \\ \langle
		\cdot, \pi_{\infty} \rangle = \epsilon_{\psi}}} \pi_{\infty} \otimes
		\chi_f(\psi). \]
	Therefore
	\[ H^{\bullet}_{(2)}(\ab, \VV_{\lambda}) \simeq \bigoplus_{\psi \in
		\Psi_{\op{disc}}^{\op{unr}, \tau}(\Sp_{2g})}
		\bigoplus_{\substack{\pi_{\infty} \in \Pi(\psi_{\infty}) \\ \langle \cdot,
		\pi_{\infty} \rangle = \epsilon_{\psi}}} H^{\bullet}((\mathfrak{g}, K),
		\pi_{\infty} \otimes V_{\lambda}) \otimes \chi_f(\psi)
	\]
	where as before $\mathfrak{g} = \mathfrak{sp}_{2g}(\CC)$ and $K = U(g)$.
\end{thm}
To be more precise the specialization of Arthur's theorem relies on
\cite[Lemma 4.1.1]{Taibi_dimtrace} and its generalization giving the Satake
parameters, considering traces of arbitrary elements of the unramified Hecke
algebra.

Thanks to \cite{AMR} for $\psi \in \Psi_{\op{disc}}^{\op{unr}, \tau}(\Sp_{2g})$
the sets $\Pi(\psi_{\infty})$ and characters $\langle \cdot, \pi_{\infty}
\rangle$ for $\pi_{\infty} \in\Pi(\psi_{\infty})$ are known to coincide
with those constructed by Adams and Johnson in \cite{AdJo}. Furthermore, the
cohomology groups
\begin{equation} \label{equ:gK_coh}
	H^{\bullet}((\mathfrak{g}, K), \pi_{\infty} \otimes V_{\lambda})
\end{equation}
for $\pi_{\infty} \in \Pi(\psi_{\infty})$ were computed explicitly in
\cite[Proposition 6.19]{VoZu}, including the real Hodge structure.

Thus in principle one can compute (algorithmically) the dimensions of the
cohomology groups $H^i_{(2)}(\ab, \VV_{\lambda})$ if the cardinalities
of the sets $O_o(\dots)$, $O_e(\dots)$ and $S(\dots)$ are known. As
explained in the previous section, for small $g$ the Euler characteristic
$e_{(2)}(\ab, \VV_{\lambda})$ can be evaluated using the trace formula. For $\pi
\in O_o(\tau) \subset \Psi_{\op{disc}}^{\op{unr}, \tau}(\Sp_{2g})$, the
contribution of $\pi$ to the Euler characteristic expanded using Theorem
\ref{thm:mult_formula} is simply $(-1)^{g(g+1)/2} 2^g \neq 0$. The contributions
of other elements of $\Psi_{\op{disc}}^{\op{unr}, \tau}(\Sp_{2g})$ to
$e_{(2)}(\ab, \VV_{\lambda})$ can be evaluated inductively, using the trace
formula and the analogue of Theorem \ref{thm:mult_formula} also for the groups
$\Sp_{2g'}$ for $g'<g$, $\SO_{4m}$ for $m \leq g/2$ and $\SO_{2m+1}$ for $m \leq
g/2$. We refer to \cite{Taibi_dimtrace} for details, and simply emphasize that
computing the contribution of some $\psi$ to the Euler characteristic is much
easier than computing all dimensions using \cite{VoZu}. For example for $\psi
\in \Psi_{\op{disc}}^{\op{unr}, \tau}(\Sp_{2g})$ we have
\begin{equation} \label{equ:simple_Euler}
	\sum_{\substack{\pi_{\infty} \in \Pi(\psi_{\infty}) \\ \langle \cdot,
	\pi_{\infty} \rangle|_{\calS_{\psi}} = \epsilon_{\psi}}} e(
	(\mathfrak{sp}_{2g}, K), \pi_{\infty} \otimes V_{\lambda}) = \pm 2^{g-r}
\end{equation}
where the integer $r$ is as in Definition \ref{def:AL_param} and the sign is
more subtle but easily computable. To sum up, using the trace formula, we
obtained tables of cardinalities for the three families of sets in Definition
\ref{def:three_sets}. See \cite{onlinelistdimtrace}.

For small $\lambda$, more precisely for $g + \lambda_1 \leq 11$, there is
another way to enumerate all elements of $\Psi_{\op{disc}}^{\op{unr},
\tau}(\Sp_{2g})$, which still relies on some computer calculations, but much
simpler ones and of a very different nature. The following is a consequence of
\cite[Th\'eor\`eme 3.3]{CheLan}. The proof uses the Riemann-Weil explicit
formula for automorphic L-functions, and follows work of Stark, Odlyzko and
Serre for zeta functions of number fields giving lower bound of their
discriminants, and of Mestre, Fermigier and Miller for L-functions of
automorphic representations. The striking contribution of \cite[Th\'eor\`eme
3.3]{CheLan} is the fact that the rank of the general linear group is not
bounded a priori, but for the purpose of the present appendix we impose
\emph{regular} infinitesimal characters.
\begin{thm} \label{thm:small_weight}
	For $w_1 \leq 11$, the only non-empty $O_o(w_1, \dots)$, $O_e(w_1, \dots)$ or
	$S(w_1, \dots)$ are the following.
	\begin{enumerate}
		\item For $g=0$, $O_o() = \{ 1 \}$.
		\item For $2 w_1 \in \{ 11, 15, 17, 19, 21 \}$, $S(w_1) = \{ \Delta_{2 w_1}
			\}$ where $\Delta_{2 w_1}$ corresponds to the unique eigenform in
			$S_{2w_1+1}(\SL_2(\CC)(\ZZ))$.
		\item $O_o(11) = \{ \Sym^2 ( \Delta_{11} ) \}$ ($\Sym^2$ functoriality was
			constructed in \cite{GelbartJacquet}),
		\item For $(2w_1, 2w_2) \in \{ (19, 7), (21, 5), (21, 9), (21, 13) \}$,
			$S(w_1, w_2) = \{ \Delta_{2w_1, 2w_2} \}$. These correspond to certain
			Siegel eigenforms in genus two and level one.
	\end{enumerate}
\end{thm}
Of course this is coherent with our tables.

As a result, for $g + \lambda_1 \leq 11$, i.e\ in
\[ \sum_{g=1}^{11} \op{card} \left\{ \lambda \,\middle|\, w(\lambda) \text{ even
and } g+ \lambda_1 \leq 11 \right\} = 1055 \]
non-trivial cases, $\Psi_{\op{disc}}^{\op{unr}, \tau}(\Sp_{2g})$ can be described
explicitly in terms of the $11$ automorphic representations of general linear
groups appearing in Theorem \ref{thm:small_weight}. In most of these cases
$\Psi_{\op{disc}}^{\op{unr}, \tau}(\Sp_{2g})$ is just empty, in fact
\[ \sum_{g + \lambda_1 \leq 11} \op{card} \Psi_{\op{disc}}^{\op{unr},
\tau}(\Sp_{2g}) = 197 \]
with $146$ non-vanishing terms.

\begin{cor} \label{cor:IHSat3V11}
	For $g=3$ and $\lambda = (1,1,0)$ we have
	\[ IH^{\bullet}( \Sat[3], \VV_{\lambda}) = 0. \]
\end{cor}
\begin{proof}
	Using Theorem \ref{thm:small_weight} and Definition \ref{def:AL_param} we see
	that $\Psi_{\op{disc}}^{\op{unr}, \tau}(\Sp_6) = \emptyset$.
\end{proof}

\begin{rem}
	Of course this result also follows from \cite{Taibi_dimtrace}. More precisely,
	without the a priori knowledge given by Theorem \ref{thm:small_weight} we have
	for $\lambda = (1,1,0)$ that $\Psi_{\op{disc}}^{\op{unr}, \tau}(\Sp_6)$ is the
	disjoint union of $O_o(4,3,1)$ with the two sets
	\[ \left\{ \pi_1 \boxplus \pi_2 \,|\, \pi_1 \in O_o(w_1),\, \pi_2 \in
	O_e(w_1', w_2') \text{ with } \{ w_1, w_1', w_2' \} = \{ 4,3,1 \} \right\}, \]
	\[ \left\{ \pi_1 \boxplus \pi_2[2] \,|\, \pi_1 \in O_o(1),\, \pi_2 \in
	S(7/2) \right\}. \]
	By Remark \ref{rem:three_sets_smallg} and vanishing of $S_{2k}(\SL_2(\ZZ))$
	for $0 < k < 6$ both sets are empty, so Corollary \ref{cor:IHSat3V11} follows
	from the computation of $e_{(2)}(\ab[3], \VV_{\lambda}) = 0$. Note that
	computationally this result is easier than Proposition \ref{prop:eA4}, since
	computing masses for $\Sp_8$ is much more work than for $\Sp_6$.
\end{rem}

If we now focus on $\lambda = 0$, we have the following classification result.

\begin{cor} \label{cor:classification_Psi_smallg}
	For $1 \leq g \leq 5$ we have $\Psi_{\op{disc}}^{\op{unr}, \rho}(\Sp_{2g}) =
	\{ [2g+1] \}$. For $6 \leq g \leq 11$, all elements of
	$\Psi_{\op{disc}}^{\op{unr}, \rho}(\Sp_{2g}) \smallsetminus \{ [2g+1] \}$ are
	listed in the following tables.

	\[
	\begin{array}{r|l|}
		g & \Psi_{\op{disc}}^{\op{unr}, \rho}(\Sp_{2g}) \smallsetminus \{ [2g+1]
		\} \\
		\hline
		6 & \Delta_{11}[2] \boxplus [9] \\
		\hline
		7 & \Delta_{11}[4] \boxplus [7] \\
    \hline
		11 & \Delta_{21}[2] \boxplus [19] \\
		   & \Delta_{21}[2] \boxplus \Delta_{11}[8] \boxplus [3] \\
		   & \Delta_{21,5}[2] \boxplus \Delta_{17}[2] \boxplus \Delta_{11}[4]
		     \boxplus [3] \\
			 & \Delta_{21}[2] \boxplus \Delta_{17}[2] \boxplus \Delta_{11}[4]
				 \boxplus [7] \\
			 & \Delta_{19}[4] \boxplus \Delta_{11}[4] \boxplus [7] \\
		   & \Delta_{21,9}[2] \boxplus \Delta_{15}[4] \boxplus [7] \\
		   & \Delta_{15}[8] \boxplus [7] \\
		   & \Delta_{21}[2] \boxplus \Delta_{17}[2] \boxplus [15] \\
		   & \Delta_{19}[4] \boxplus [15] \\
		   & \Delta_{11}[10] \boxplus \Sym^2(\Delta_{11}) \\
		   & \Delta_{17}[6] \boxplus [11] \\
		   & \Delta_{21}[2] \boxplus \Delta_{15}[4] \boxplus [11] \\
		   & \Delta_{21,13}[2] \boxplus \Delta_{17}[2] \boxplus [11] \\
	\end{array}
	\quad
	\begin{array}{r|l|}
		g & \Psi_{\op{disc}}^{\op{unr}, \rho}(\Sp_{2g}) \smallsetminus \{ [2g+1]
		\} \\
		\hline
		8 & \Delta_{11}[6] \boxplus [5] \\
				& \Delta_{15}[2] \boxplus \Delta_{11}[2] \boxplus [9] \\
				& \Delta_{15}[2] \boxplus [13] \\
		\hline
		9 & \Delta_{11}[8] \boxplus [3] \\
				& \Delta_{17}[2] \boxplus \Delta_{11}[4] \boxplus [7] \\
				& \Delta_{17}[2] \boxplus [15] \\
				& \Delta_{15}[4] \boxplus [11] \\
		\hline
		10 & \Delta_{19,7}[2] \boxplus \Delta_{15}[2] \boxplus \Delta_{11}[2]
					\boxplus [5] \\
				 & \Delta_{19}[2] \boxplus \Delta_{11}[6] \boxplus [5] \\
				 & \Delta_{11}[10] \boxplus [1] \\
				 & \Delta_{19}[2] \boxplus [17] \\
				 & \Delta_{19}[2] \boxplus \Delta_{15}[2] \boxplus \Delta_{11}[2]
					\boxplus [9] \\
				 & \Delta_{15}[6] \boxplus [9] \\
				 & \Delta_{17}[4] \boxplus \Delta_{11}[2] \boxplus [9] \\
				 & \Delta_{19}[2] \boxplus \Delta_{15}[2] \boxplus [13] \\
				 & \Delta_{17}[4] \boxplus [13] \\
	\end{array}
	\]
\end{cor}

In the next section we will recall and make explicit the description by
Langlands and Arthur of $L^2$-cohomology in terms of $\psi_{\infty}$, which is
simpler than using \cite{VoZu} directly but imposes to work with the group
$\GSp_{2g}$ (which occurs in a Shimura datum) instead of $\Sp_{2g}$.

Let us work out the simple case of $\psi = [2g+1] \in
\Psi_{\op{disc}}^{\op{unr}, \rho}(\Sp_{2g})$ directly using \cite[Proposition
6.19]{VoZu}, using their notation, in particular $\mathfrak{g} = \mathfrak{k}
\oplus \mathfrak{p}$ (complexification of the Cartan decomposition of
$\mathfrak{g}_0 = \op{Lie}(\Sp_{2g}(\RR))$ and $\mathfrak{p} = \mathfrak{p}^+
\oplus \mathfrak{p}^-$ (decomposition of $\mathfrak{p}$ according to the action
of the center of $K$, which is isomorphic to $U(1)$). In this case
$\Pi(\psi_{\infty}) = \{ 1 \}$, $\mathfrak{u} = 0$ and $\mathfrak{l} =
\mathfrak{sp}_{2g}$, so
\[ H^{2k}((\mathfrak{g}, K), 1) = H^{k,k}((\mathfrak{g}, K), 1) \simeq \Hom_K
	\left(\bigwedge^{2k} \mathfrak{p}, \CC \right)
\]
and the dimension of this space is the number of constituents in the
multiplicity-free representation $\bigwedge^k( \mathfrak{p}^+)$. One can show
(as for any Hermitian symmetric space) that this number equals the number of
elements of length $p$ in $W(\Sp_{2g}) / W(K) \simeq W(\Sp_{2g}) / W(\GL_g)$.
A simple explicit computation that we omit shows that this number equals the
number of partitions of $k$ as a sum of distinct integers between $1$ and $n$.
In conclusion,
\[ \sum_i T^i \dim H^i( (\mathfrak{sp}_{2n}, K), 1) = \prod_{k=1}^n (1+T^{2k}).
\]

Combining the first part of Corollary \ref{cor:classification_Psi_smallg} and
this computation we obtain Theorem \ref{teo:Intersectionsmall}.
\begin{thm} \label{thm:IH_Sat_taut_smallg}
	For $1 \leq g \leq 5$ we have $IH^{\bullet}(\Sat, \QQ) \simeq R_g$ as graded
	vector spaces over $\QQ$.
\end{thm}

We can sharpen Theorem \ref{thm:stable_IH} in the particular case of level one,
although what we obtain is not a stabilization result (see the remark after the
theorem).

\begin{thm} \label{thm:sharper_stab_IH_Sat_level_one}
	For $g \geq 2$, $\lambda$ a dominant weight for $\Sp_{2g}$ and $k < 2g-2$,
	\[ IH^k(\Sat, \VV_{\lambda}) =
		\begin{cases}
			R_g^k & \text{ if } \lambda = 0 \\
			0 & \text{ otherwise }
		\end{cases} \]
		where $R^k_g$ denotes the degree $k$ part of $R_g$, $u_i$ having degree
		$2i$.
\end{thm}
\begin{proof}
	For $\psi \in \Psi_{\op{disc}}^{\op{unr}, \tau}(\Sp_{2g})$ different from
	$[2g+1]$, one sees easily from the construction in \cite{AdJo} that the
	trivial representation of $\Sp_{2g}(\RR)$ does not belong to
	$\Pi(\psi_{\infty})$. So using Zucker's conjecture, Borel-Casselman and
	Arthur's multiplicity formula, we are left to show that for $\psi \in
	\Psi_{\op{disc}}^{\op{unr}, \tau}(\Sp_{2g}) \smallsetminus \{ [2g+1] \}$, for
	any $\pi_{\infty} \in \Pi(\psi_{\infty})$, $H^{\bullet}((\mathfrak{g}, K),
	\pi_{\infty} \otimes V_{\lambda})$ vanishes in degree less than $2g-2$. We can
	read this from \cite[Proposition 6.19]{VoZu}, and we use the notation from
	this paper. Let $\theta$ be the Cartan involution of $\Sp_{2g}(\RR)$
	corresponding to $K$, so that $\mathfrak{p} = \mathfrak{g}^{-\theta}$. The
	representation $\pi_{\infty}$ is constructed from a $\theta$-stable parabolic
	subalgebra $\mathfrak{q} = \mathfrak{l} \oplus \mathfrak{u}$ of
	$\mathfrak{g}$, where $\mathfrak{l}$ is also $\theta$-stable. We will show
	that $\dim \mathfrak{u} \cap \mathfrak{p} \geq 2g-2$. The (complex) Lie
	algebra $\mathfrak{l}$ is isomorphic to $\prod_j \mathfrak{gl}(a_j+b_j) \times
	\mathfrak{sp}_{2c}$ with $c + \sum_j a_j+b_j = g$. The action of the
	involution $\theta$ on the factor $\mathfrak{gl}(a_k+b_k)$ is such that the
	associated real Lie algebra is isomorphic to $\mathfrak{u}(a_k,b_k)$. Using
	notation of Definition \ref{def:AL_param}, the integer $c$ equals $(d_0-1)/2$.
	Since $\psi \neq [2g+1]$ we have $r \geq 1$ and this implies that $c \leq g-2$
	(this is particular to level one, in arbitrary level one would simply get
	$c<g$). We have
	\[ 2 \dim \mathfrak{u} \cap \mathfrak{p} + \dim \mathfrak{l} \cap \mathfrak{p}
	= \dim \mathfrak{p} = g(g+1) \]
	since $\mathfrak{l}$, $\mathfrak{u}$ and its opposite Lie algebra with respect
	to $\mathfrak{l}$ are all stable under $\theta$. We compute
	\[ \dim \mathfrak{l} \cap \mathfrak{k} = \sum_j (a_j^2+b_j^2) + c^2 \]
	and so
	\[ \dim \mathfrak{l} \cap \mathfrak{p} = \dim \mathfrak{l} - \dim \mathfrak{l}
	\cap \mathfrak{k} = 2 \sum_j a_jb_j + c(c+1) .\]
	We get
	\[ \dim \mathfrak{u} \cap \mathfrak{p} = \frac{g(g+1)}{2} - \frac{c(c+1)}{2} -
	\sum_j a_jb_j. \]
	We have
	\[ \sum_j a_jb_j \leq (\sum_j a_j)(\sum_j b_j) \leq \frac{(g-c)^2}{4} \]
	which implies
	\[ \dim \mathfrak{u} \cap \mathfrak{p} \geq \frac{g(g+1)}{2} -
	\frac{c(c+1)}{2} - \frac{(g-c)^2}{4}. \]
	The right hand side is a concave function of $c$, thus its maximal value for
	$c \in \{0, \dots, g-2\}$ is $g(g+1)/2 - \min(g^2/4, (g-2)(g-1)/2+1)$. For
	integral $g \neq 3$ one easily checks that this equals $g(g+1)/2 -
	(g-2)(g-1)/2 - 1 = 2g-2$. For $g=3$ we get $2g-2-1/4$, and $\lceil 2g-2-1/4
	\rceil = 2g-2$.
\end{proof}

\begin{rem}\label{rem:ihsatcomments}
	\begin{enumerate}
		\item For $k \geq g$ even the surjective map $\QQ[\lambda_1, \lambda_3,
			\dots]^k \rightarrow R_g^k$ has non-trivial kernel, so in the range $g
			\leq k < 2g-2$ we do \emph{not} get stabilization.
		\item One can show that for the trivial system of coefficients (i.e.\
			$\lambda=0$) the bound in Theorem \ref{thm:sharper_stab_IH_Sat_level_one}
			is sharp for even $g \geq 6$, is not sharp for odd $g$ (i.e.\
			$IH^{2g-2}(\Sat, \QQ) = R_g^{2g-2}$) but that for odd $g \geq 9$ we have
			$IH^{2g-1}(\Sat, \QQ) \neq 0$. For even $g \geq 6$ and odd $g \geq 9$ this
			is due to $\psi$ of the form $\pi[2] \boxplus [2g-3]$ where $\pi \in
			S(g-1/2)$ corresponds to a weight $2g$ eigenform for $\SL_2(\ZZ)$.
		\item Of course for non-trivial $\lambda$ one can improve on this result,
			e.g.\ for $\lambda_1 > \dots > \lambda_g > 0$ we have vanishing in degree
			$\neq g(g+1)/2$ (see \cite[Theorem 5]{saper1} and \cite[Theorem
			5.5]{LiSchwermer} for a vanishing result for ordinary cohomology). If we
			only assume $\lambda_g > 0$, this forces $c=0$ in the proof and we obtain
			vanishing in degree less than $(g^2+2g)/4$.
		\item An argument similar to the proof of Theorem
			\ref{thm:sharper_stab_IH_Sat_level_one} can be used to show the same
			result for $k<g$ in arbitrary level. It seems likely that one could extend
			the sharper bound in Theorem \ref{thm:sharper_stab_IH_Sat_level_one} to
			certain deeper levels, e.g.\ Iwahori level at a finite number of primes.
	\end{enumerate}
\end{rem}

There is also a striking consequence of \cite[Proposition 6.19]{VoZu} (and
\cite[Lemma 9.1]{Kottwitz_AnnArbor}) that was observed in \cite{MorelSuh},
namely the fact that any $\psi$ only contributes in degrees of a certain parity.
This implies the dimension part in the following proposition which is a natural
first step towards the complete description in the next section.

\begin{prop} \label{prop:decomp_IH}
	There is a canonical decomposition
	\[ \QQ^{\op{real}} \otimes_{\QQ} IH^{\bullet}(\Sat, \VV_{\lambda}) =
		\bigoplus_{\psi \in \Psi_{\op{disc}}^{\op{unr}, \tau}(\Sp_{2g})}
	\QQ^{\op{real}} \otimes_{E(\psi)} H_{\psi}^{\bullet} \]
	where $H_{\psi}^{\bullet}$ is a graded vector space of total dimension
	$2^{n-r}$ ($r$ as in Definition \ref{def:AL_param}) over the totally real
	number field $E(\psi)$, endowed with
	\begin{enumerate}
		\item for any $n \geq 0$, a pure Hodge structure of weight $n$ on
			$H_{\psi}^n$, inducing a bigrading $\CC \otimes_{E(\psi)} H_{\psi}^n =
			\bigoplus_{p+q=n} H_{\psi}^{p,q}$,
		\item a linear operator $L : \RR \otimes_{E(\psi)} H_{\psi}^{\bullet}
			\rightarrow \RR \otimes_{E(\psi)} H_{\psi}^{\bullet}$ mapping
			$H_{\psi}^{p,q}$ to $H_{\psi}^{p+1,q+1}$ and such that for any $0 < n \leq
			g(g+1)/2$,
			\[ L^n : \RR \otimes_{E(\psi)} H_{\psi}^{g(g+1)/2-n} \rightarrow \RR
			\otimes_{E(\psi)} H_{\psi}^{g(g+1)/2+n} \]
			is an isomorphism.
	\end{enumerate}
	This decomposition is $\calH_f^{\op{unr}}(\Sp_{2n})$-equivariant, the action
	on $H_{\psi}$ being by the character $\chi_f(\psi)$.
\end{prop}
\begin{proof}
	Recall that by Zucker's conjecture (\cite{LoijZucker}, \cite{saperstern},
	\cite{LooijengaRapoport}) we have a Hecke-equivariant isomorphism
	\[ IH^{\bullet}(\Sat, \VV_{\lambda}) \otimes_{\QQ} \RR \simeq
	H^{\bullet}_{(2)}(\ab, \VV_{\lambda} \otimes_{\QQ} \RR). \]
	By Theorem \ref{thm:mult_formula} there are graded vector spaces $H_{\psi}$
	that can be defined over $E_{\psi}$ such that the left hand side is isomorphic
	to the right hand side. By Theorem \ref{thm:mult_one} each summand on the
	right hand side can be cut out using Hecke operators, so the decomposition is
	canonical and the $E(\psi)$-structure on $H_{\psi}$ is canonical as well. We
	endow $\RR \otimes_{\QQ} IH^{\bullet}(\Sat, \VV_{\lambda}) \simeq
	H^{\bullet}_{(2)}(\ab, \RR \otimes_{\QQ} \VV_{\lambda})$ with the real Hodge
	structure given by Hodge theory on $L^2$-cohomology of the non-compact
	K\"ahler manifold $\ab$. There is a natural Lefschetz operator $L$ given by
	cup-product with the K\"ahler form. It commutes with Hecke operators and one
	can check that $L$ is $i$ times the operator $X$ defined on p.60 of
	\cite{Arthur_unip}. The hard Lefschetz property of $L$ is known both in
	$L^2$-cohomology and $(\mathfrak{g}, K)$-cohomology. It follows from
	\cite[Theorem 1.5]{MorelSuh} that any $\psi$ contributes in only one parity,
	so the claim about $\dim_{E(\psi)} H_{\psi}$ follows from
	\eqref{equ:simple_Euler}.
\end{proof}

If $o$ is any $\op{Gal}(\QQ^{\op{real}} / \QQ)$-orbit in
$\Psi_{\op{disc}}^{\op{unr}, \tau}(\Sp_{2g})$, $\bigoplus_{\psi \in o} H_{\psi}$
is naturally defined over $\QQ$ and endowed with an action of a quotient of
$\calH_f^{\op{unr}}(\Sp_{2g})$ which is a finite totally real field extension
$E(o)$ of $\QQ$, and elements of $o$ correspond bijectively to $\QQ$-embeddings
$E(o) \rightarrow \QQ^{\op{real}}$.

\begin{rem}
	There is also a rational Hodge structure defined on intersection cohomology
	groups thanks to Morihiko Saito's theory of mixed Hodge modules, but
	unfortunately it is not known whether the induced real Hodge structure
	coincides with the one defined using $L^2$ theory (see \cite[\S
	5]{HarrisZucker3}). Similarly, there is another natural Lefschetz operator
	acting on the cohomology $IH^{\bullet}(\Sat, \VV_{\lambda})$ (using the first
	Chern class of an ample line bundle on $\Sat$), and it does not seem obvious
	that it coincides (up to a real scalar) with the K\"ahler operator $L$ above,
	although it could perhaps be deduced from arguments as in \cite[\S
	16.6]{GoreskyPardon}.
\end{rem}

\subsection{Description in terms of Archimedean Arthur-Langlands parameters}
\label{sec:app_Shimura}

Langlands and Arthur (\cite{Arthur_unip}, \cite{Arthur_L2_exp}) gave a
conceptually simpler point of view on the Hodge structure with Lefschetz
operator on $L^2$-cohomology. This applies to Shimura varieties and so one would
have to work with the reductive group $\GSp_{2g}$ instead of $\Sp_{2g}$, since
only $\GSp_{2g}$ is part of a Shimura datum, $(\GSp_{2g}, \HH_g \sqcup
\overline{\HH_g})$. Very roughly, the idea of this description for a Shimura
datum $(G,X)$ is that for $K_1$ the stabilizer in $G(\RR)$ of a point in $X$,
representations of $G(\RR)$ in an Adams-Johnson packet are parametrized by
certain cosets in $W(G,T)/W(K_1,T)$ for a maximal torus $T$ of $G(\RR)$
contained in $K_1$, and $W(K_1,T)$ is also identified with the stabilizer of the
cocharacter $\mu : \GL_1(\CC) \rightarrow G(\CC)$ obtained from the Shimura
datum. This cocharacter can be seen as an extremal weight for an irreducible
algebraic representation $r_{\mu}$ of $\widehat{G}$ and $r_{\mu}$ is
\emph{minuscule}, i.e.\ its weights form a single orbit under the Weyl group of
$\widehat{G}$, which is identified with $W(G,T)$.

In the case of $\GSp_{2g}$ the Langlands dual group is $\widehat{\GSp_{2g}} =
\GSpin_{2g+1}$ and $r_{\mu}$ is a spin representation. Morphisms taking values
in a spin group cannot simply be described as self-dual linear representations.
For this reason we do not have substitutes for Arthur-Langlands parameters for
$\GSp_{2g}$ constructed using automorphic representations of general linear
groups (that is the analogue of \ref{def:AL_param} for $\GSp_{2g}$ and arbitrary
level), and no precise multiplicity formula yet. Bin Xu \cite{Xu} obtained a
multiplicity formula in many cases, but his work does not cover the case of
non-tempered Arthur-Langlands parameters that is typical when $\lambda=0$. For
example all parameters appearing in Corollary
\ref{cor:classification_Psi_smallg} are non-tempered. Fortunately in level one
it turns out that we can simply formulate the result in terms of $\Sp_{2g}$.
This is in part due to the fact that, letting $K_1 = \RR_{>0} \Sp_{2g}(\RR)
\subset \GSp_{2g}(\RR)$, the natural map
\begin{equation} \label{equ:isom_Sp_GSp}
	\ab = \Sp_{2g}(\QQ) \backslash \Sp_{2g}(\AQ) / K \Sp_{2g}(\widehat{\ZZ})
	\rightarrow \GSp_{2g}(\QQ) \backslash \GSp_{2g}(\AQ) / K_1
	\GSp_{2g}(\widehat{\ZZ})
\end{equation}
is an isomorphism. This is a special case of a more general principle in level
one, see \S 4.3 and Appendix B in \cite{ChenevierRenard} for a conceptual
explanation. Since the cohomology of the intermediate extension to $\Sat$ of
$\VV_{\lambda}$ vanishes when the weight $w(\lambda) := \lambda_1 + \dots +
\lambda_g$ is odd, we will be able to formulate the result using
$\Spin_{2g+1}(\CC)$ instead of $\GSpin_{2g+1}(\CC)$.

Let $g \geq 1$ and $\lambda$ a dominant weight for $\Sp_{2g}$, as usual let
$\tau = \lambda + \rho$. Consider
\[ \psi = \psi_0 \boxplus \dots \boxplus \psi_r = \pi_0[d_0] \boxplus \dots
	\boxplus \pi_r[d_r] \in \Psi_{\op{disc}}^{\op{unr}, \tau}(\Sp_{2g}) \]
as in Definition \ref{def:AL_param}. First we recall how to equip $\CC
\otimes_{E(\psi)} H_{\psi}^{\bullet}$ with a continuous semisimple linear action
$\rho_{\psi}$ of $\CC^{\times} \times \SL_2(\CC)$. This action will be trivial
on $\RR_{>0} \subset \CC^{\times}$ by construction.
\begin{enumerate}
	\item We let $z \in \CC^{\times}$ act on $H_{\psi}^{p,q}$ by multiplication by
		$(z/|z|)^{q-p}$.
	\item There is a unique algebraic action of $\SL_2(\CC)$ on
		$\CC \otimes_{E(\psi)} H_{\psi}^{\bullet}$ such that the action of
		$\begin{pmatrix} 0 & 1 \\ 0 & 0 \end{pmatrix} \in \mathfrak{sl}_2$ is given by
		the Lefschetz operator $L$ and the diagonal torus in $\SL_2(\CC)$ preserves
		the grading on $\CC \otimes_{E(\psi)} H_{\psi}^{\bullet}$. Explicitly, by
		hard Lefschetz we have that for $t \in \GL_1$, $\op{diag}(t, t^{-1}) \in
		\SL_2$ acts on $H_{\psi}^i$ by multiplication by $t^{g(g+1)/2-i}$. This
		algebraic action is defined over $\RR$, and if we knew that $L$ is rational
		it would even be defined over $E(\psi)$.
	\item These actions commute and we obtain
		\[ \rho_{\psi} : \CC^{\times} \times \SL_2(\CC) \longrightarrow \GL(\CC
		\otimes_{E(\psi)} H_{\psi}). \]
\end{enumerate}
The dimension $g(g+1)/2$ being fixed, we see that the isomorphism class of the
real Hodge structure with Lefschetz operator $(\RR \otimes_{E(\psi)} H_{\psi},
L)$ determines and is determined by the isomorphism class of $\rho_{\psi}$. In
fact they are both determined by the Hodge diamond of $\RR \otimes_{E(\psi)}
H_{\psi}^{\bullet}$.

To state the description of these isomorphism classes in terms of
Arthur-Langlands parameters we need a few more definitions. For $i \in \{0,
\dots, r\}$ let $m_i$ be the product of $d_i$ with the dimension of the standard
representation of $\widehat{G_{\pi_i}}$, so that $2g+1 = \sum_{i=0}^r m_i$. Let
$\calM_{\psi_0} = \SO_{m_0}(\CC)$. For $1 \leq i \leq r$ let $(\calM_{\psi_i},
\tau_{\psi_i})$ be a pair such that $\calM_{\psi_i} \simeq \SO_{m_i}(\CC)$ and
$\tau_{\psi_i}$ is a semisimple element in the Lie algebra of $\calM_{\psi_i}$
whose image via the standard representation has eigenvalues 
\[ \pm (w_1^{(i)} + \frac{d_i-1}{2} - j) , \dots, \pm ( w_{g_i}^{(i)} +
\frac{d_i-1}{2} -j) \text{ for } 0 \leq j \leq d_i-1. \]
As in Definition \ref{def:Lpsi} the point of this definition is that the group
of automorphisms of $(\calM_{\psi_i}, \tau_{\psi_i})$ is the adjoint group of
$\calM_{\psi_i}$, because $\tau_{\psi_i}$ is not invariant under the outer
automorphism of $\calM_{\psi_i}$. Note that $\calM_{\psi_i}$ is semisimple since
$m_i \neq 2$.

Let $\calM_{\psi} = \prod_{0 \leq i \leq r} \calM_{\psi_i}$. There is a natural
embedding $\iota_{\psi} : \calM_{\psi} \rightarrow \SO_{2g+1}(\CC)$. Up to
conjugation by $\calM_{\psi}$ there is a unique morphism $f_{\psi} :
\calL_{\psi} \times \SL_2(\CC) \rightarrow \calM_{\psi}$ such that
\begin{enumerate}
	\item $\iota_{\psi} \circ f_{\psi}$ is conjugated to $\dot{\psi}$, which
		implies that $f_{\psi}$ is an algebraic morphism,
	\item the differential of $f_{\psi}$ maps $( (\tau_{\pi_i})_{1 \leq i \leq r},
		\op{diag}(\frac{1}{2}, - \frac{1}{2}))$ to $\tau_{\psi} :=
		(\tau_{\psi_i})_{0 \leq i \leq r}$.
\end{enumerate}
We can conjugate $\dot{\psi}$ in $\SO_{2g+1}(\CC)$ so that $\iota_{\psi} \circ
f_{\psi} = \dot{\psi}$, so we assume this equality from now on. The centralizer
of $\iota_{\psi}$ in $\SO_{2g+1}(\CC)$ coincides with $\calS_{\psi}$. Let
\[ f_{\psi, \infty} = f_{\psi} \circ \left((\varphi_{\pi_{i, \infty}})_{0 \leq i
\leq r}, \op{Id}_{\SL_2(\CC)} \right) : W_{\RR} \times \SL_2(\CC) \rightarrow
\calM_{\psi}. \]
Condition (ii) above is equivalent to $\op{ic}(f_{\psi, \infty}) = \tau_{\psi}$.
We have $\psi_{\infty} = \iota_{\psi} \circ f_{\psi, \infty}$.
		
Let $\calM_{\psi, \op{sc}} = \prod_{i=0}^r \calM_{\psi_i, \op{sc}} \simeq
\prod_{i=0}^r \Spin_{m_i}(\CC)$ be the simply connected cover of $\calM_{\psi}$.
Let $\spin_{\psi_0}$ be the spin representation of $\calM_{\psi_0, \op{sc}}$, of
dimension $2^{(m_0-1)/2}$. Let $\calM_{\psi_i, \op{sc}}$ be the simply connected
cover of $\calM_{\psi_i}$. The group $\calM_{\psi_i, \op{sc}}$ has two half-spin
representations $\spin^{\pm}_{\psi_i}$, distinguished by the fact that the
largest eigenvalue of $\spin^+_{\psi_i}(\tau_{\psi_i})$ is greater than that of
$\spin^-_{\psi_i}(\tau_{\psi_i})$. They both have dimension $2^{m_i/2-1}$. Let
$\calL_{\psi, \op{sc}} = \prod_{i=0}^r (\widehat{G_{\pi_i}})_{\op{sc}}$ be the
simply connected cover of $\calL_{\psi}$, a product of spin and symplectic
groups. There is a unique algebraic lift $\widetilde{f_{\psi}} : \calL_{\psi,
\op{sc}} \times \SL_2(\CC) \rightarrow \calM_{\psi, \op{sc}}$ of $f_{\psi}$.
There is a unique algebraic lift $\widetilde{\iota_{\psi}} : \calM_{\psi,
\op{sc}} \rightarrow \Spin_{2g+1}(\CC)$ of $\iota_{\psi}$ and it has finite
kernel. The pullback of the spin representation of $\Spin_{2g+1}(\CC)$ via
$\widetilde{\iota_{\psi}}$ decomposes as
\begin{equation} \label{equ:decomp_spin}
	\spin_{\psi_0} \otimes \bigoplus_{(\epsilon_i)_i \in \{ \pm \}^r}
	\spin^{\epsilon_1}_{\psi_1} \otimes \dots \otimes \spin^{\epsilon_r}_{\psi_r}.
\end{equation}
Let us be more specific. It turns out that the preimage $\calS_{\psi, \op{sc}}
\simeq (\ZZ/2\ZZ)^{r+1}$ of $\calS_{\psi}$ in $\Spin_{2g+1}(\CC)$ commutes with
$\widetilde{\iota_{\psi}}(\calM_{\psi, \op{sc}})$. This is specific to conductor
one, i.e.\ level $\Sp_{2g}(\widehat{\ZZ})$. Thus the spin
representation of $\Spin_{2g+1}(\CC)$ restricts to a representation of
$\calM_{\psi, \op{sc}} \times \calS_{\psi, \op{sc}}$, and
\eqref{equ:decomp_spin} is realized by decomposing into isotypical components
for $\calS_{\psi, \op{sc}}$. More precisely, the non-trivial element of the
center of $\Spin_{2g+1}(\CC)$ is mapped to $-1$ in the spin representation, and
there is a natural basis $(s_i)_{1 \leq i \leq r}$ of $\calS_{\psi}$ over
$(\ZZ/2\ZZ)^r$ and lifts $(\tilde{s}_i)_{1 \leq i \leq r}$ in $\calS_{\psi,
\op{sc}}$ such that in each factor of \eqref{equ:decomp_spin} $\tilde{s}_i$ acts
by $\epsilon_i$.

For $0 \leq i \leq r$ there is a unique continuous lift
$\widetilde{\varphi_{\pi_{i, \infty}}} : \CC^{\times} \rightarrow
(\widehat{G_{\pi_i}})_{\op{sc}}$ of \emph{the restriction} $\varphi_{\pi_{i,
\infty}}|_{\CC^{\times}}$. One could lift morphisms from $W_{\RR}$ but the lift
would not be unique. Finally we can define
\[ \widetilde{f_{\psi, \infty}} = \widetilde{f_{\psi}} \circ
	\left((\widetilde{\varphi_{\pi_{i, \infty}}})_{0 \leq i \leq r},
	\op{Id}_{\SL_2(\CC)} \right) : \CC^{\times} \times \SL_2(\CC) \rightarrow
\calM_{\psi}. \]

\begin{thm} \label{thm:desc_Hpsi_AL}
	For $g \geq 1$, $\lambda$ a dominant weight for $\Sp_{2g}$ and
	\[ \psi = \psi_0 \boxplus \dots \boxplus \psi_r \in
	\Psi_{\op{disc}}^{\op{unr}, \tau}(\Sp_{2g}) \]
	we have an isomorphism of continuous semisimple representations of
	$\CC^{\times} \times \SL_2(\CC)$:
	\[ \rho_{\psi} \simeq \left( \spin_{\psi_0} \otimes \spin^{u_1}_{\psi_1}
	\otimes \dots \otimes \spin^{u_r}_{\psi_r} \right) \circ \widetilde{f_{\psi,
	\infty}} \]
	where $u_1, \dots, u_r \in \{ +, - \}$ can be determined explicitly (see
	\cite{Taibi_Ag} for details).
\end{thm}
\begin{proof}
	This is essentially a consequence of \cite[Proposition 9.1]{Arthur_unip} and
	Arthur's multiplicity formula, but we need to argue that in level one the
	argument goes through with the multiplicity formula for $\Sp_{2g}$ (Theorem
	\ref{thm:mult_formula}) instead. This is due to two simple facts. Firstly, the
	group $K_1$ in \cite[\S 9]{Arthur_unip} is simply $\RR_{>0} \times K$, and
	this implies that for $\pi_{\infty}$ a unitary irreducible representation of
	$\GSp_{2g}(\RR) / \RR_{>0}$,
	\[ H^{\bullet}((\mathfrak{gsp}_{2g}, K_1), \pi_{\infty} \otimes V_{\lambda}) =
		H^{\bullet}((\mathfrak{sp}_{2g}, K), \pi_{\infty}|_{\Sp_{2g}(\RR)}
	\otimes V_{\lambda} ) \]
	where on the left hand side $V_{\lambda}$ is seen as an algebraic
	representation of $\op{PGSp}_{2g}$ (since we can assume that $w(\lambda)$
	even). Secondly, as we observed above the preimage $\calS_{\psi, \op{sc}}$ of
	$\calS_{\psi}$ in $\Spin_{2g+1}(\CC)$ still commutes with
	$\widetilde{\iota_{\psi}}(\calM_{\psi, \op{sc}})$, making the representation
	$\sigma_{\psi}$ of \cite[\S 9]{Arthur_unip} well-defined. This second fact is
	particular to the level one case.
\end{proof}

To conclude, if we know $\Psi_{\op{disc}}^{\op{unr}, \tau}(\Sp_{2g})$, making
the decomposition in Proposition \ref{prop:decomp_IH} completely explicit boils
down to computing signs $(u_i)_{1 \leq i \leq r}$ and branching in the following
cases:
\begin{enumerate}
	\item for the morphism $\Spin_{2a+1} \times \SL_2 \rightarrow
		\Spin_{(2a+1)(2b+1)}$ lifting the representation
		\[ \op{Std_{\SO_{2a+1}}} \otimes \Sym^{2b}(\op{Std}_{\SL_2}) : \SO_{2a+1}
		\times \SL_2 \rightarrow \SO_{(2a+1)(2b+1)} \]
		and the spin representation of $\Spin_{(2a+1)(2b+1)}$,
	\item for $\Spin_{4a} \times \SL_2 \rightarrow \Spin_{4a(2b+1)}$ and both
		half-spin representations,
	\item for $\Sp_{2a} \times \SL_2 \rightarrow \Spin_{4ab}$ and both half-spin
		representations.
\end{enumerate}

For example one can using Corollary \ref{cor:classification_Psi_smallg},
Proposition \ref{prop:decomp_IH} and Theorem \ref{thm:desc_Hpsi_AL} one can
explicitly compute $IH^{\bullet}(\ab)$ for all $g \leq 11$.

\begin{exa} \label{exa:IH_as_spin}
	\begin{enumerate}
		\item For any $g \geq 1$ and $\psi = [2g+1]$, the group $\calL_{\psi}$ is
			trivial and up to a shift we recover the graded vector space $R_g$ as the
			composition of the spin representation of $\Spin_{2g+1}$ composed with the
			principal morphism $\SL_2 \rightarrow \Spin_{2g+1}$, graded by weights of
			a maximal torus of $\SL_2$.
		\item Consider $g = 6$ and $\psi = \Delta_{11}[2] \boxplus [9]$. For
			$\psi_0$ we have $\spin_{\psi_0} \circ \widetilde{f_{\psi_0}} \simeq
			\nu_{11} \oplus \nu_5$, where as before $\nu_d$ denotes the irreducible
			$d$-dimensional representation of $\SL_2$. For $\psi_1 = \Delta_{11}[2]$
			we have $\calL_{\psi_1} = \Sp_2(\CC)$, $\calM_{\psi_1} = \SO_4(\CC)$,
			$\spin^+_{\psi_1} \circ \widetilde{f_{\psi_1}} \simeq \op{Std}_{\Sp_2}
			\otimes 1_{\SL_2}$ and $\spin^-_{\psi_1} \circ \widetilde{f_{\psi_1}}
			\simeq 1_{\Sp_2} \otimes \nu_2$. It turns out that $u_1 = -$, so
			$\rho_{\psi}|_{\CC^{\times}}$ is trivial and
			\[ \rho_{\psi}|_{\SL_2} \simeq (\nu_{11} \oplus \nu_5 ) \otimes \nu_2
			\simeq \nu_{12} \oplus \nu_{10} \oplus \nu_6 \oplus \nu_4. \]
			Thus $H_{\psi}^{\bullet}$ has primitive cohomology classes in degrees
			$10,12,16,18$ (a factor $\nu_d$ contributes a primitive cohomology class
			in degree $g(g+1)/2-d+1$). Surprisingly, these classes are all Hodge,
			i.e.\ they belong to $H_{\psi}^{2k} \cap H_{\psi}^{k,k}$, despite the fact
			that the parameter $\psi$ is explained by a non-trivial motive over $\QQ$
			(attached to $\Delta_{11}$).
		\item Consider $g = 7$ and $\psi = \Delta_{11}[4] \boxplus [7]$. Again
			$\calL_{\psi, \op{sc}} \simeq \Sp_2(\CC)$. For $\psi_0 = [7]$ we have
			$\spin_{\psi_0} \circ \widetilde{f_{\psi_0}} \simeq \nu_7 \oplus 1$. For
			$\psi_1 = \Delta_{11}[4]$ we have $\calL_{\psi_1} = \Sp_2(\CC)$,
			$\calM_{\psi_1} = \SO_8(\CC)$, $\spin^+ \simeq \Sym^2(\op{Std}_{\Sp_2})
			\oplus \nu_5$ and $\spin^- \simeq \op{Std}_{\Sp_2} \otimes \nu_4$. Here
			$u_1 = +$, and we conclude
			\begin{align*}
				\rho_{\psi} & \simeq (\nu_7 \oplus 1) \otimes (\Sym^2(\op{Std}_{\Sp_2})
				\oplus \nu_5) \\
				& \simeq \Sym^2(\op{Std}_{\Sp_2}) \otimes (\nu_7 \oplus 1) \oplus
				\nu_{11} \oplus \nu_9 \oplus \nu_7 \oplus \nu_5^{\oplus 2} \oplus \nu_3.
			\end{align*}
			In this example, as in general, we would love to know that the above
			formula is valid for the \emph{rational} Hodge structure
			$H_{\psi}^{\bullet}$, replacing $\Sym^2(\op{Std}_{\Sp_2})$ by
			$\Sym^2(M)(11)$ where $M$ is the motivic Hodge structure associated to
			$\Delta_{11}$. In \cite{Taibi_Ag} this (and generalizations) is proved at
			the level of $\ell$-adic Galois representations.
	\end{enumerate}
\end{exa}

Forgetting the Hodge structure, the graded vector space $H_{\psi}^{\bullet}$ is
completely described by the restriction of $\rho_{\psi}$ to $\SL_2(\CC)$. The
Laurent polynomial $T^{-g(g+1)/2} \sum_{k=0}^{g(g+1)} T^k \dim H_{\psi}^k$ can
easily be computed by taking the product over $0 \leq i \leq r$ of the following
Laurent polynomials (with choice of signs as in Theorem \ref{thm:desc_Hpsi_AL}).
Denote $x = (1, \op{diag}(T,T^{-1})) \in \CC^{\times} \times \SL_2(\CC)$.

\begin{enumerate}
	\item For $\psi_0 = \pi_0[2d+1]$ with $\pi_0 \in O_o(w_1, \dots, w_m)$, we
		have
		\[ \left( \spin_{\psi_0} \circ \widetilde{f_{\psi, \infty}} \right)(x) = 2^m
		\prod_{j=1}^d (T^{-j} + T^j)^{2m+1}. \]
		\item For $\psi_i = \pi_i[2d+1]$ with $\pi_i \in O_e(w_1, \dots, w_{2m})$ we
			have
			\[ \left( \spin^{\pm}_{\psi_i} \circ \widetilde{f_{\psi, \infty}}
				\right)(x) = 2^{2m-1} \prod_{j=1}^d (T^{-j} + T^j)^{4m}. \]
		\item For $\psi_i = \pi_i[2d]$ with $\pi_i \in S(w_1, \dots, w_m)$ we have
			\begin{multline*}
				\left( \spin^{\pm}_{\psi_i} \circ \widetilde{f_{\psi, \infty}}
				\right)(x) = \\
				\frac{1}{2} \left( \prod_{j=1}^d (2+T^{2j-1}+T^{1-2j})^m \pm
				\prod_{j=1}^d (2-T^{2j-1}-T^{1-2j})^m \right).
			\end{multline*}
\end{enumerate}

\bibliographystyle{plain}
\bibliography{biblio-abelian}

\end{document}